%% file: main.tex
\documentclass[letterpaper,11pt]{scrartcl}

\input{g_preamble}

\title{Fixed Parameter Multi-Objective Evolutionary Algorithms for the $W$-Separator Problem}

\author[*]{Samuel~Baguley}
\author[*]{Tobias~Friedrich}
\author[**]{Aneta~Neumann}
\author[**]{Frank~Neumann}
\author[*]{Marcus~Pappik}
\author[*]{Ziena~Zeif}

\affil[*]{\normalsize
	\{firstname.lastname\}@hpi.de\\
	\vspace*{\baselineskip}
	Hasso Plattner Institute, University of Potsdam
}

\affil[**]{\normalsize
	\{firstname.lastname\}@adelaide.edu.au\\
	\vspace*{\baselineskip}
	University of Adelaide
}

\begin{document}
	\maketitle
	\setcounter{page}{0}
	\thispagestyle{empty}

	\begin{abstract}
		Parameterized analysis provides powerful mechanisms for obtaining fine-grained insights into different types of algorithms.
		In this work, we combine this field with evolutionary algorithms and provide parameterized complexity analysis of evolutionary multi-objective algorithms for the $W$-separator problem, which is a natural generalization of the vertex cover problem.
		The goal is to  remove  the  minimum  number of  vertices such  that  each  connected component in the resulting graph has at most $W$ vertices.
		We provide different multi-objective formulations involving two or three objectives that provably lead to fixed-parameter evolutionary algorithms with respect to the value of an optimal solution $OPT$ and $W$.
		Of particular interest are kernelizations and the reducible structures used for them.
		We show that in expectation the algorithms make incremental progress in finding such structures and beyond.
		The current best known kernelization of the $W$-separator uses linear programming methods and requires a non-trivial post-process to extract the reducible structures.
		We provide additional structural features to show that evolutionary algorithms with appropriate objectives are also capable of extracting them.
		Our results show that evolutionary algorithms with different objectives guide the search and admit fixed parameterized runtimes to solve or approximate (even arbitrarily close) the $W$-separator problem.
	\end{abstract}
	
	\section{Introduction}
	\input{g_intro}
	
	\section{Preliminaries}\label{sec::prelim}
	\input{g_prelim}
	
	\section{Analysis for degree-based fitness function}\label{sec::f1}
	\input{g_fitness_1}
	
		\section{Analysis for LP-based fitness function}\label{sec::f2}
	\input{g_fitness_2}
	
	\section{Approximations}\label{sec::apx}
	\input{g_apx}
	
	\section{Conclusion}
	
	In this work, we studied the behavior of evolutionary algorithms with different multi-objective fitness functions for the $W$-separator problem from the perspective of parameterized complexity.
	More precisely, we investigated the running time of such evolutionary algorithms depending on the problem parameter $\opt + W$.
	Our analysis was based on properties of reducible structures, showing that, given a suitable fitness function, the evolutionary algorithm tends to reduce the given instance along these structures.
	Once this is done, the running time for either obtaining an arbitrarily close approximation or an exact solution is tractable with respect to the problem parameter.
	In particular, this shows that evolutionary algorithms solve the $W$-separator problem in expectation in FPT-time for the parameter $\opt + W$.

	\bibliography{literature}
	
	\appendix
	
	\section{Omitted proofs of Section~\ref{sec::prelim}}\label{appendix::prelim}
	\input{appendix_prelim}
	
	\section{Omitted proofs of Section~\ref{sec::f1}}\label{appendix::f1}
	\input{appendix_f1}
	
	\section{Omitted proofs of Section~\ref{sec::f2}}\label{appendix::f2}
	\input{appendix_f2}
	
	\section{Omitted proofs of Section~\ref{sec::apx}}\label{appendix::apx}
	\input{appendix_apx}
	
\end{document}

%% file: g_preamble.tex
\usepackage{mathtools}
\usepackage{enumitem} 
\usepackage{bbm}
\usepackage{geometry}
\usepackage{hyperref}
\usepackage[capitalize]{cleveref}
\usepackage[utf8x]{inputenc} 
\usepackage[dvipsnames]{xcolor}
\usepackage{amsmath,amssymb,amsthm}
\setlist[enumerate,1]{label={(\roman*)}}
\usepackage{float}
\usepackage{authblk} 

\usepackage[ruled,linesnumbered]{algorithm2e}

\usepackage{xparse} 

\usepackage[ruled,linesnumbered]{algorithm2e}

\usepackage{caption}
\usepackage{subcaption}
\usepackage{dsfont}

\usepackage{thmtools,thm-restate}

\newtheorem{theorem}{Theorem}
\newtheorem{definition}[theorem]{Definition}

\newtheorem{lemma}[theorem]{Lemma}
\newtheorem{corollary}[theorem]{Corollary}

\newtheorem*{theorem-non}{Theorem}

\NewDocumentCommand{\mathOrText}{m}
{%
	\ensuremath{#1}\xspace%
}

\DeclareDocumentCommand{\Pr}{m o}
{
	\mathOrText{ \mathds{P}\left[ #1 \IfNoValueF{#2}{\;\middle\vert\; #2} \right]}
}
\DeclareDocumentCommand{\E}{m o}
{
	\mathOrText{ \mathds{E}\left[ #1 \IfNoValueF{#2}{\;\middle\vert\; #2} \right]}
}
\DeclareDocumentCommand{\filtration}{o}
{
	\mathOrText{ \mathcal{F}\IfNoValueF{#1}{_{#1}}}
}

\usepackage{todonotes}

\def\N{\mathds{N}}
\def\R{\mathds{R}}

\def\O{\mathcal{O}}
\def\opt{\textsc{OPT}}

\def\P{\mathcal{P}}
\def\lp{\text{LP}}
\def\opt{\text{OPT}}

\def\lpPrim{\text{LP}}

\def\Q{\mathcal{Q}}
\def\B{\mathcal{B}}

\def\AB{A \cup B}
\def\oAB{\overline{A \cup B}} 
\def\oB{\overrightarrow{\B}}
\def\ora{\overrightarrow}
\def\aall{A_{\texttt{all}}}
\def\ball{B_\texttt{all}}

\def\alg1+1{\texttt{(1+1)~EA}\xspace}
\def\algGlobalSemo{\texttt{Global Semo}\xspace}
\def\algGlobalSemoAlt{\texttt{Global Semo Alt}\xspace}

\def\algAltMut{Alternative Mutation Operator\xspace}


%% file: g_intro.tex

Parameterized analysis of algorithms~\cite{downey2012parameterized} provides a way of understanding the working behaviour of algorithms via their dependence on important structural parameters for NP-hard problems. This technique of fine-grained analysis allows for insights into which parameters make a problem hard. When analyzing heuristic search methods such as evolutionary algorithms, a parameterized runtime analysis allows for runtime bounds not just dependent on the given input size but also in terms of parameters that measure the difficulty of the problem. This is particularly helpfule for understanding heuristic search methods which are usually hard to analyze in a rigorous way.

Parameterized analysis of evolutionary algorithms has been carried out for several important combinatorial optimization problems (see \cite{DBLP:series/ncs/0001S20} for an overview). 
The first analysis was for the classical vertex cover problem~\cite{DBLP:journals/algorithmica/KratschN13} which is the prime problem in the area of parameterized complexity.
Following that, problems such as the maximum leaf spanning problem~\cite{DBLP:conf/ppsn/KratschLNO10}, the Euclidean traveling salesperson problem~\cite{DBLP:journals/ec/SuttonNN14} and parameterized settings of 
makespan scheduling~\cite{DBLP:conf/ppsn/SuttonN12} were considered. More recently, both the closest string problem~\cite{DBLP:journals/algorithmica/Sutton21} and
jump and repair operators have been analyzed in the parameterized setting~\cite{DBLP:conf/foga/BransonS21}.
A crucial aspect of the parameterized analysis of evolutionary algorithms (and algorithms in general) is the ability of the considered approaches to obtain a kernelization for the considered problems. A kernel here refers to a smaller sub-problem whose size is polynomially bounded in the size of the given parameter(s). As the size is bounded, brute-force methods or random sampling can then be applied to obtain an optimal solution.

A small subset of vertices that disconnect a graph is usually called a vertex separator.
In terms of successful divide-and-conquer or parallel processing strategies, such separators are one of the most powerful tools for developing efficient graph algorithms.
This generality and its broad applicability have made the study of separators a rich and
active field of research, see for example the book by Rosenberg and Heath~\cite{DBLP:books/daglib/0007445}, or the line of research
initiated by the seminal work of Lipton and Tarjan \cite{doi:10.1137/0136016} on separators in planar graphs.
Numerous different types of separator structures have emerged over the past couple of decades.
In this paper, we address the problem of decomposing a graph into small pieces - with respect to a parameter $W$ - by removing the smallest possible set of vertices. 
More formally, given a graph $G=(V,E)$ and a parameter $W \in \N$, the goal is to remove the minimum number of vertices such that each connected component in the resulting graph has at most $W$ vertices.
The problem is called the \emph{$W$-separator problem} - also known in the literature as the \emph{component order connectivity problem} or \emph{$\alpha$-balanced separator problem}, where $\alpha  \in (0,1)$ and $W=\alpha|V|$.
An equivalent view of this problem is to cover or hit every connected subgraph of size $W+1$ with the minimum number of vertices.
In particular, $W=1$ corresponds to covering all edges, showing that the $W$-separator problem is a natural generalization of the vertex cover problem.

In this paper, we generalize the results obtained in \cite{DBLP:journals/algorithmica/KratschN13} for the vertex cover problem to the more general $W$-separator problem.
Precisely, we investigate the $W$-separator problem on multi-objective evolutionary algorithms and show that in expectation they admit fixed parameter runtimes with respect to the value of an optimal solution $\opt$ and $W$.
It is unlikely that such runtimes can be achieved by considering $\opt$ or $W$ alone.
Indeed, $W=1$ corresponds to a hard problem, which shows that $W$ (alone) is not a suitable parameter.
For the parameter $\opt$, the problem is $W[1]$-hard even when restricted to split graphs~\cite{DBLP:journals/algorithmica/DrangeDH16}.
These lower bounds lead to the study of parameterization by $W+\opt$.
The best known algorithm with respect to these parameters finds an optimal solution in time $n^{O(1)} \cdot 2^{\O(\log(W) \cdot \opt)}$~\cite{DBLP:journals/algorithmica/DrangeDH16}.
Unless the exponential time hypothesis fails, the authors prove that this running time is tight up to constant factors, i.e., there is no algorithm that solves the problem in time $n^{\O(1)} \cdot 2^{o(\opt \cdot \log(W))}$.
For kernelizations with respect to the parameters $\opt$ and $W$, the best known polynomial algorithm achieves a kernel of size $3 W \cdot \opt$~\cite{DBLP:conf/esa/Casel0INZ21}.
A kernel of size $2 W \cdot \opt$ is provided in~\cite{DBLP:conf/iwpec/KumarL16} in a runtime of $n^{\O(1)} \cdot 2^{\O(W)}$ by using linear programming methods.
That is, for the cover problem, ie.~$W=1$, they obtain a $2 \cdot \opt$ size-kernel, showing that under the assumption the unique games conjecture is true, $2W \cdot \opt$ is the best kernel we can hope for.
Finally, the best known approximation algorithm also uses linear programming methods and has a gap guarantee of $\O(\log(W))$ with a running time of $2^{\O(W)} \cdot n^{\O(1)}$~\cite{DBLP:journals/corr/Lee16c}. 
They also showed that the superpolynomial dependence on $W$ may be needed to achieve a polylogarithmic approximation.

\paragraph{\textbf{Our Contribution:}}
Of particular interest in our work are kernelizations and the reducible structures used for them.
We show that in expectation the algorithms make incremental progress in finding such structures and beyond.
Compared to the vertex cover problem, kernelization algorithms that are linear in $\opt$ for the $W$-separator are more complicated (cf.~\cite{DBLP:journals/corr/Xiao16b,DBLP:conf/esa/Casel0INZ21,DBLP:conf/iwpec/KumarL16}).
The current best known kernelization of the $W$-separator uses linear programming methods and requires a non-trivial post-process to extract the reducible structures~\cite{DBLP:conf/iwpec/KumarL16}.
The challenge in this paper is to show that natural objectives and simple as well as problem-independent mutations are also capable of extracting them.
To this end, we add additional structural features to the reducible structures used in~\cite{DBLP:conf/iwpec/KumarL16}.
Essentially, our results show that evolutionary algorithms with different objectives guide the search and admit fixed parameterized runtimes to solve or approximate (even arbitrarily close) the $W$-separator problem.
 
The different runtimes are given in this paper in terms of the number of iterations, but the tractability with respect to the considered parameters also applies when we include search point evaluations.
In the following, we roughly describe the runtimes achieved with respect to the search point evaluations for exact and approximate solutions, where all results are given in expectation.
We consider simple and problem-independent evolutionary algorithms in combination with three different multi-objective fitness functions.
The first consists of relatively simple calculations to evaluate the search points and allows us to achieve a running time of $n^{\O(1)} \cdot 2^{\O(\opt^2 \cdot W^2)}$ to find an optimal solution. 
For the second and third fitness functions, stronger objectives are used in the sense of applying linear programming methods.
We prove that with such evaluations the optimal solution can be found in time $n^{\O(1)} \cdot 2^{\O(\opt \cdot W)}$.  
Moreover, depending on the choice of an $\varepsilon \in [0,1)$ we obtain solutions arbitrary close to an optimal one, where the according algorithm is tractable with respect to the parameters $\opt$ and $W$.
As usual, the larger $\varepsilon$, the worse the gap guarantee, but with better running time, where $\varepsilon=0$ corresponds to the above running time in finding an optimal solution.
This result shows that we can hope for a gradual progress until an optimal solution is reached.

Finally, our results show that in expectation evolutionary algorithms are not far away from the problem-specific ones, with the literature review showing that up to constant factors the evolved algorithms are close to the lower bounds for the $W$-separator problem.	

%% file: g_prelim.tex

\paragraph{\textbf{Graph Terminology}}
We begin with a brief introduction to the graph terminology we use in this paper.
Let $G=(V,E)$ be a graph.
For a subgraph $G'=(V',E')$ of $G$ we use $V(G')$ and $E(G')$ to denote $V'$ and $E'$, respectively.
We define the \emph{size} of a subgraph $G' \subseteq G$ as the number of its vertices, where we denote the size of $G$ by $n$.
For $v \in V$ we define $N(v)$ as its neighborhood, and $d(v)$ as the degree of $v$.
For a vertex subset $V' \subseteq V$ we define $G[V']$ as the induced subgraph of $V'$, $G-V' := G[V \setminus V']$ and $N(V') := \left(\bigcup_{v \in V'} N(v)\right) \setminus V'$.
Finally, in the context of this work, we also use directed graphs in the sense of flow networks, where we move the corresponding terminology to the appendix next to the proofs.

\paragraph{\textbf{Parameterized Terminology}}
We use the standard terminology for parameterized complexity, which is also used, for example, in \cite{downey2012parameterized,fomin2019kernelization}.
A \emph{parameterized problem} is a decision problem with respect to certain instance parameters.
Let $I$ be an instance of a parameterized problem with an instance parameter $k$, usually given as a pair $(I,k)$.
If for each pair $(I,k)$ there exists an algorithm that solves the decision problem in time $f(k) \cdot |I|^c$, where $f$ is a computable function and $c$ is a constant, then the parameterized problem is \emph{fixed-parameter tractable}.
We say $(I,k)$ is a \emph{yes-instance} if the answer to the decision problem is positive, otherwise we say $(I,k)$ is a \emph{no-instance}.

Of particular interest in this work are kernelizations, which can be roughly described as formalized preprocessings.
More formally, given an instance $(I,k)$ of a parameterized problem, a polynomial algorithm is called a kernelization if it maps any $(I,k)$ to an instance $(I',k')$ such that $(I',k')$ is a yes-instance if and only if $(I,k)$ is a yes-instance, $|I'| \leq g(k)$, and $k' \leq g'(k)$ for computable functions $g,g'$.

The idea of parameterized complexity can be extended by combining multiple parameters.
That is, if we consider an instance $I$ with parameters $k_1, \dots, k_m$, then we are interested in algorithms that solve the corresponding decision problem in a runtime of $f(k_1, \dots, k_m) \cdot |I|^c$, where $f$ is a computable function and $c$ is a constant.
We refer to runtimes that satisfy this type of form as \emph{FPT-times}. 

\paragraph{\textbf{Problem Statement and Objectives}}
First we introduce the \textit{$W$-separator problem}.
Given is a graph $G=(V,E)$ and two positive integers $k$ and $W$.
The challenge is to find a vertex subset $V' \subseteq V$, such that $V'$ has cardinality at most $k$ and the removal of $V'$ in $G$ leads to a graph that contains only connected components of size at most $W$.
The minimization problem is to find $V'$ with the smallest cardinality, where we denote the optimal objective value by $\opt$.
Note that we can reformulate the problem statement by demanding that $V'$ intersects with each connected subgraph of $G$ of size $W+1$.
In the case $W=1$ a separator needs to cover each edge, which shows that the $W$-separator problem is a natural generalization of the well-known vertex cover problem.

In terms of evolutionary algorithms, a solution to the $W$-separator problem can be represented in a bit sequence of length $n$. 
Each vertex has value zero or one, where one stands for the vertex being part of the $W$-separator.
Let $\{0,1\}^{n}$ be our solution space.
We work with multi-objective evolutionary algorithms, which evaluate each search point $X \in \{0,1\}^{n}$ using a fitness function $f \colon \{0, 1\}^{n} \to \R^m$ with $m$ different objectives.
The goal is to minimize each of the objectives.
Denote by $f^i(X)$ the $i$-th objective, evaluated at a search point $X$.
For two search points $X_1$ and $X_2$, we say $X_1$ \emph{weakly dominates} $X_2$ if $f^i(X_1) \leq f^i(X_2)$ for every $i \in [m]$, where $[m]$ is defined as the set $\{1,\dots,m\}$.
In this case, we simply write $f(X_1) \leq f(X_2)$.
If additionally $f(X_1) \neq f(X_2)$, then we say that $X_1$ \emph{dominates} $X_2$. 
We distinguish between pareto optimal search points $X$ and vectors $f(X)$. 
A pareto optimal search point is a search point that is not even weakly dominated by any other search point, whereas a pareto optimal vector is not dominated by any other vector.
That is, if $f(X_1)$ is a pareto optimal vector, then there can be a vector $X_2 \neq X_1$ with $f(X_2) = f(X_2)$, whereas if $X_1$ is a pareto optimal search point, then there is no search point $X_2 \neq X_1$ with $f(X_1) = f(X_2)$.

For some fitness functions we investigate, we use a linear program to evaluate the search points.
Let $G=(V,E)$ be an instance of the $W$-separator and let $y_v \in \{0,1\}$ be a variable for each $v \in V$.
An integer program (IP) that solves the $W$-separator problem can be formulated as follows:
\begin{align*}
	\min &\sum_{v \in V} y_v\\
	&\sum_{v \in S} y_v \geq 1, \forall S \subseteq V \colon  |S| = W+1 \text{ and } G[S] \text{ is connected }
\end{align*}
We will consider the relaxed version of the IP by allowing fractional solutions and consider the corresponding linear program (LP).
That is, instead of $y_v \in \{0,1\}$ we have $y_v \geq 0$ for all $v \in V$.
In the rest of this paper we will call it the \emph{$W$-separator LP}.
We define $\lpPrim(G')$ for a subgraph $G' \subseteq G$ as the objective of the $W$-separator LP with $G'$ as input graph.
If we put every connected subgraph of size $W+1$ as constraint in the LP formulation of the $W$-separator, then we end up with a running time of $n^{\O(W)}$.
However, as mentioned already in Fomin et.~al.~\cite{fomin2019kernelization}~(Section 6.4.2) finding an optimal solution for the LP can be sped up to a running time of $2^{\O(W)} n^{\O(1)}$.
Roughly speaking, the idea is to use the ellipsoid method with separation oracles to solve the linear program, where the separation oracle uses a method called color coding that makes it tractable in $W$.

Next, we define few additional terms before we get to the multiobjective fitness functions.
Let $X \in \{0,1\}^n$ be a search point.
For $v \in V$ we define $x_v \in \{0,1\}$ as the corresponding value in the bit-string $X$. 
We denote by $X_1 \subseteq V$ the vertices with value one.
We define $u(X)$ as the set of vertices that are in components of size at least $W+1$ after the removal of $X_1$ in $G$.
The function $u(X)$ can be interpreted as the uncovered portion of the graph with respect to the vertices $X_1$.
The fitness functions we work with are the following:

\begin{itemize}
	\item $f_1(X) := \left(|X_1|, |u(X)|, -\sum_{v \in X_1} d(v)\right)$
	\item $f_2(X) := \left(|X_1|, |u(X)|, \lpPrim(G[u(X)])\right)$
	\item $f_3(X) := \left(|X_1|, \lpPrim(G[u(X)])\right)$
\end{itemize}

As the names suggest, we use \emph{one-objective}, \emph{uncovered-objective}, \emph{degree-objective} and \emph{LP-objective} to denote $|X_1|,|u(X)|, -\sum_{v \in X_1} d(v)$ and $\lpPrim(G[u(X)])$ respectively.
Note that the fitness $f_3$ is same as $f_2$ without the uncovered-objective.
Furthermore, we use $*$ to denote that an objective can be chosen arbitrarily, for instance in $(|X_1|,*,-\sum_{v \in X_1} d(v))$ the uncovered-objective $u(X)$ is arbitrarily.

\paragraph{\textbf{Algorithms}}
We proceed by presenting the algorithms that we study.
All of them are based on \algGlobalSemo (see Algorithm \ref{alg:GlobalSemo}), which maintains a population $\P \subseteq \{0, 1\}^{n}$ of $n$-dimensional bit strings.

\begin{algorithm}
	\caption{\algGlobalSemo} \label{alg:GlobalSemo}
	
	Choose $X \in \{0,1\}^n$ uniformly at random
	
	$\P \gets \{X\}$
	
	\While{stopping criterion not met}
	{
		Choose $X  \in \P$ unformly at random
		
		$Y \gets$ flip each bit of $X$ independently with probability $1/n$
		\label{algGlobSemo::Mutation}

		If $Y$ is not dominated by any other search point in $\P$, include $Y$ into $\P$ and delete all other bit strings $Z \in \P$ which are weakly dominated by $X$ from $\P$, i.e.~those with $f(Y) \leq f(Z)$.
	}	
\end{algorithm}

\begin{algorithm}
	\caption{\algAltMut} \label{alg:altMut}
	
	Choose $b \in \{0,1,2\}$ uniformly at random
	
	\uIf{$b=2$ \textbf{and} $u(X) \neq \varnothing$}
	{
		$Y \gets$ for $v \in u(X)$ flip each bit $x_v$ with probability $1/2$ 
	}
	
	\uElseIf{$b=1$ \textbf{and} $X_1 \neq \varnothing$}
	{
		$Y \gets$ for $v \in X_1$ flip each bit $x_v$ with probability $1/2$ 
	}
	
	\Else
	{
		$Y \gets$ flip each bit of $X$ independently with probability $1/n$
	}
\end{algorithm}

We define the Algorithm \algGlobalSemoAlt similarly to the Algorithm \algGlobalSemo (see Algorithm~\ref{alg:GlobalSemo})  
with the difference that the mutation in line~\ref{algGlobSemo::Mutation} is exchanged by \algAltMut (see Algorithm~\ref{alg:altMut}). 
The following two lemmas will be useful throughout the whole paper.
Their proofs are similar to some appearing in~\cite{DBLP:journals/algorithmica/KratschN13}, and due to space constraints we have moved them to the appendix (Section~\ref{appendix::prelim}).

\begin{lemma}
	\label{lemma::singleFlipAndBoundedPop}
	Let $\P \neq \varnothing$ be a population for the fitness functions $f_1$ and $f_2$.
	In the Algorithms \algGlobalSemo and \algGlobalSemoAlt,
	selecting a certain search point $X \in \P$ has probability $\Omega(1/n^2)$, and additionally flipping only one single bit in it has probability $\Omega(1/n^3)$.
\end{lemma}

Let $0^n$ be the search point that contains only zeroes.
Note that once $0^n$ is in the population it is pareto optimal for all fitness functions because of the one-objective.

\begin{lemma}
	\label{lemma::zeroSol}
	Using the the fitness functions $f_1$ or $f_2$, the expected number of iterations of \algGlobalSemo or \algGlobalSemoAlt until the population $\P$ contains the search point $0^n$ is $\O(n^3 \log n)$.
\end{lemma}

The following lemmas are proven analogously to \cref{lemma::singleFlipAndBoundedPop,lemma::zeroSol} by observing that the worst-case bounds on the population size decrease by a factor of $n$ when using fitness function $f_3$ instead of $f_1$ or $f_2$.

\begin{lemma}
	\label{corollary::singleFlipAndBoundedPop}
	Let $\P \neq \varnothing$ be a population for the fitness function $f_3$.
	In the Algorithms \algGlobalSemo and \algGlobalSemoAlt,
	selecting a certain search point $X \in \P$ has probability $\Omega(1/n)$, and additionally flipping only one single bit in it has probability $\Omega(1/n^2)$.
\end{lemma}

\begin{lemma}
	\label{corollary::zeroSol}
	Using the the fitness function $f_3$, the expected number of iterations of \algGlobalSemo or \algGlobalSemoAlt until the population $\P$ contains the search point $0^n$ is $\O(n^2 \log n)$.
\end{lemma}

%% file: g_fitness_1.tex

In this section we investigate the fitness $f_1$ on \algGlobalSemoAlt.
We will prove that the algorithm finds an optimal $W$-separator in expectation in FPT-runtime with the parameters $\opt$ and $W$.
Recall that the parameter $k$ in the decision variant of the $W$-separator asks for a $W$-separator of size at most $k$.
A more general variant, known as \textit{weighted component order connectivity problem}, was studied in~\cite{DBLP:journals/algorithmica/DrangeDH16} by Drange et al.
They achieve a $\O(k^2 W + W^2 k)$ vertex-kernel, which also holds for the $W$-separator problem.

\begin{theorem}[\cite{DBLP:journals/algorithmica/DrangeDH16}, Theorem~15]
	\label{thm::degreeKernel}
	The $W$-separator admits a kernel with at most $kW(k+W)+k$ vertices, where $k$ is the solution size.
\end{theorem}

Essentially, they use the following \emph{reduction rule}: as long as there is a vertex with degree greater than $k + W$, the vertex is included in the solution set and may be removed from the instance.

It is not difficult to see that this vertex must be included in the solution, since otherwise we would have to take more than $k$ vertices from its neighborhood to get a feasible solution.
After using this reduction rule exhaustively each vertex in the reduced instance has degree at most $k+W$.
Consequently, in the reduced instance, each vertex of a $W$-separator is connected to at most $k+W$ connected components after its removal, where each of those components has size at most $W$.  
A simple calculation provides finally the vertex-kernel stated in \cref{thm::degreeKernel}.

Now, we make use of the degree-objective from $f_1$ to find a search point that selects those vertices which can be safely added to an optimal solution according to the reduction rule.

\begin{lemma}
	\label{lemma::reducedInstance_f1}
	Using the fitness function $f_1$, the expected number of iterations of 
	\algGlobalSemoAlt where the population $\P$ contains a solution $X$ in which for all $u \in u(X)$ and for all $v \in X_1$ we have $d(u) \leq \opt+W$ and $d(v) > \opt+W$ is bounded by $\O(n^3(\opt + \log n))$. 
\end{lemma}

With \cref{lemma::reducedInstance_f1} in hand we can upper bound the expected number of iterations that \algGlobalSemoAlt takes to find an optimal $W$-separator with respect to the fitness $f_1$.
Note that the uncovered-objective of $f_1$ ensures that the algorithms \algGlobalSemoAlt converge to a feasible solution and that a search point $X$ with $f_1(X) = (\opt, 0, *)$ corresponds to an optimal $W$-separator.   

\begin{theorem}
	\label{thm::fitness1Opt}
	Using the fitness function $f_1$, the expected number of iterations of \algGlobalSemoAlt until it finds a minimum $W$-separator in $G=(V,E)$ is upper bounded by $\O\left(n^3(\opt + \log n) + n^2 \cdot 2^q \right)$, where $q=\opt \cdot W(\opt+W)+\opt$.
\end{theorem}

%% file: g_fitness_2.tex

In this section we investigate $f_2$ on \algGlobalSemoAlt.
The main result of this section is the following theorem.

\begin{theorem}
	\label{thm::OptSolf2}
	Let $G=(V,E)$ be an instance of the $W$-separator problem.
	Using the fitness function $f_2$, the expected number of iterations of \algGlobalSemoAlt until an optimal solution is sampled is upper bounded by $\O(n^3(\log n + \opt) + n^2 \cdot 4^{\opt \cdot W} )$.
\end{theorem}
 
First we give a brief overview of a reducible structure concerning the $W$-separator problem associated with the objectives in the fitness function $f_2$.   
The structure we will use is commonly known as crown decomposition.
Roughly speaking, it is a division of the set of vertices into three parts consisting of a crown, a head, and a body, with the head separating the crown from the body.
Under certain conditions concerning the crown and head vertices, which we will clarify in a moment, it is possible to show that there exists an optimal $W$-separator which contains the head vertices and reduces the given instance by removing the crown vertices.
Recall that the parameter $k$ in the decision variant of the $W$-separator asks for a $W$-separator of size at most $k$.
Kumar and Lokshtanov~\cite{DBLP:conf/iwpec/KumarL16} provide such a reducible structure and state that it is in a graph as long as the size of it is greater than $2kW$.
The structure is called a (strictly) reducible pair and consists of crown and head vertices.

For an instance $G=(V,E)$ of the $W$-separator problem we say that $Y=\{y_v \in \R_{\geq 0}\}_{v \in V}$ is a \emph{fractional $W$-separator} of $G$ if $Y$ is a feasible solution according to the LP formulation of the $W$-separator problem.
It is not difficult to see that the objective of any optimal fractional $W$-separator is smaller than $\opt$, i.e.~$\lpPrim(G) \leq \opt$.
In principle, the LP objective is useful for finding a strictly reducible pair, since the head vertices in an optimal fractional W separator must have value one.
Unfortunately, it is unknown whether each vertex that has value one in an optimal fractional $W$-separator is part of an optimal solution.
This leads to the challenge of filtering out the right vertices, where the uncovered-objective - and in particular the structural properties of a strictly reducible pair - come into play.

\paragraph{\textbf{Reducible Structure of the $\boldsymbol W$-Separator Problem}}
In the following, we briefly summarize the definitions and theorems of \cite{DBLP:conf/esa/Casel0INZ21,fomin2019kernelization,DBLP:conf/iwpec/KumarL16}. 
For a vertex set $B \subseteq V$, denote by $\B$ the partitioning of $B$ according to the connected components of $G[B]$.

\begin{definition}[(strictly) reducible pair]
\label{def::RedPair}
For a graph $G=(V,E)$, a pair $(A,B)$ of vertex disjoint subsets of $V$ is a \emph{reducible pair} if the following conditions are satisfied:
\begin{itemize}
	\item $N(B) \subseteq A$.
	\item
	The size of each $C \in \B$ is at most $W$.
	\item There is an assignment function $g \colon \B \times A \to \N_0$, such that
	\begin{itemize}
		\item for all $C \in \B$ and $a \in A$, if $g(C,a) \ne 0$, then $a \in N(C)$
		\item for all $a \in A$ we have $\sum_{C \in \B} g(C,a) \geq 2W-1$,
		\item for all $C \in \B$ we have  $\sum_{a \in A} g(C,a) \leq |C|$,
	\end{itemize} 
\end{itemize} 
In addition, if there exists an $a \in A$ such that $\sum_{Q \in \B} g(C,a) \geq 2 W$, then $(A,B)$ is a \emph{strictly reducible pair}.
\end{definition}

Next, we explain roughly the idea behind a reducible pair $(A,B)$.
The head and crown vertices correspond to $A$ and $B$, respectively.
That is, we want $A$ to be part of our $W$-separator, and if that is the case, then no additional vertex from $B$ is required to be in the solution since the components $C \in \B$ are isolated after removing $A$ from $G$ with $|C| \leq W$. 
Let $G=(V,E)$ be a graph. 
We say that $P_1, \dots, P_m \subseteq V$ is a \emph{($W+1$)-packing} if for all $i,j \in [m]$ with $i \neq j$ the induced subgraph $G[P_i]$ is connected, $|P_i| \geq W+1$, and $P_i \cap P_j = \varnothing$.
Note that for a $W$-separator $S \subseteq V$, it holds that $S \cap P_i \neq \varnothing$ for all $i \in [m]$.
Thus, the size of a ($W+1$)-packing is a lower bound on the number of vertices needed for a $W$-separator.

\begin{lemma}[\cite{DBLP:conf/iwpec/KumarL16}, Lemma 17]
\label{lemma::packingReduciblePair}
Let $(A,B)$ be a reducible pair in $G$.
There is a ($W+1$)-packing $P_1, \dots, P_{|A|}$ in $G[A \cup B]$, such that $|P_{i} \cap A| = 1$ for all $i \in [|A|]$.
\end{lemma}

Essentially, \cref{lemma::packingReduciblePair} provides a lower bound of $|A|$ vertices for a $W$-separator in $G[A \cup B]$.
On the other hand, $A$ is a $W$-separator of $G[A \cup B]$ while $A$ separates $B$ from the rest of the graph.
This properties basically admits the following theorem.

\begin{theorem}[\cite{DBLP:conf/iwpec/KumarL16}, Lemma 18]
\label{thm::saveReduction}
Let $(G,k)$ be an instance of the $W$-separator problem, and $(A,B)$ be a reducible pair in $G$.
$(G,k)$ is a yes-instance if and only if $(G - (A \cup B), k - |A|)$ is a yes-instance.
\end{theorem}

Finally, we clarify why a strictly reducible pair exists if the size of $G$ is larger than $2kW$.
To do so, we make use of a lemma derivable from  \cite{DBLP:conf/esa/Casel0INZ21,DBLP:conf/iwpec/KumarL16}.
A proof is given in the appendix~(\cref{appendix::f2}). 

\begin{lemma}
\label{lemma::balancedExpansionStrictlyReducibleExist}
Let $G=\left(A \cup B, E\right)$ be a graph and $W \in \N_0$.
Let $\B$ be the connected components given as vertex sets of $G[B]$, where for each $C \in \B$ we have $|C| \leq W$ and no $C \in \B$ is isolated, i.e.~$N(C) \neq \varnothing$.
If $|B| \geq (2W-1)|A|+1$, then there exists a non-empty strictly reducible pair $(A',B')$, where $A' \subseteq A$ and $B' \subseteq B$.
\end{lemma}

We conclude the preliminary section with a lemma that connects strictly reducible pairs with the size of the graph.

\begin{lemma}[\cite{fomin2019kernelization}, Lemma~6.14]
\label{lemma::existenceReducible}
Let $(G,k)$ be an instance of the $W$-separator problem, such that each component in $G$ has size at least $W+1$.
If $|V| > 2Wk$ and $(G,k)$ is a yes-instance, then there exists a strictly reducible pair $(A,B)$ in $G$. 
\end{lemma}

\paragraph*{\textbf{Running time analysis}} 
Let $(A,B)$ be a strictly reducible pair. 
We say $(A,B)$ is a minimal strictly reducible pair if there does not exist a strictly reducible pair $(A',B')$ with $A' \subset A$ and $B' \subseteq B$.
Clearly, it can happen that reducible pairs arises after a reduction is executed.
Therefore, we say $(A_1,B_1), \dots, (A_m,B_m)$ is a \emph{sequence of minimal strictly reducible pairs} if for all $i \in [m]$ the tuple $(A_i,B_i)$ is a minimal strictly reducible pair in $G - \bigcup_{j=1}^{i-1} A_j$.
Note that the definition of such a sequence implies that those tuples are pairwise disjoint, i.e., $(A_i \cup B_i) \cap (A_j \cup B_j) = \varnothing$ for all $i,j \in [m]$ with $i \neq j$.
The proof of \cref{thm::OptSolf2} can essentially be divided into three phases:
\begin{enumerate}
\item
\label{part1}
Let $(A_1,B_1), \dots, (A_m,B_m)$ be a sequence of minimal strictly reducible pairs in $G$, such that $G - \bigcup_{i\in [m]} A_i$ contains no minimal strictly reducible pair. 
The first phase is to show that after a polynomial number of iterations of \algGlobalSemoAlt with fitness $f_2$, a search point $X \in \{0, 1\}^n$ exists in the population $\P$, such that $\lpPrim(G) = |X_1| + \lpPrim(G[u(X)])$ and there is a fractional optimal $W$-separator $Y = \{y_v \in \R_{\geq 0}\}_{v \in u(X)}$ with $y_v < 1$ for each $v \in V$.
We will prove that in this case $G[u(X)]$ contains no strictly reducible pair, and that because of the equality relation $\lpPrim(G) = |X_1| + \lpPrim(G[u(X)])$ all the head vertices $A_i$ for $i \in [m]$ are in $X_1$.
That is, there is an optimal $W$-separator which contains a subset of $X_1$.
\item
\label{part2}
The second phase is to filter $\bigcup_{i=1} A_i$ from $|X_1|$ so that we obtain a search point $X'$ that selects only those as 1-bits.
Once an $X$ as described in Phase~{\ref{part1}} is guaranteed to be in the population, the algorithm \algGlobalSemoAlt takes in expectation FPT-time to reach $X'$.
Finally, it is important that $X'$ remains in the population once we have found it.
We show this by taking advantage of the structural properties of a reducible pair in combination with the uncovered-objective.
\item
\label{part3}
For the last phase, we know by \cref{lemma::existenceReducible} already that $u(X')$ has size at most $2 \cdot \opt \cdot W$. 
Once we ensure that $X'$ is in $\P$ and stays there, we prove that \algGlobalSemoAlt finds in expectation an optimal solution in FPT-time.
\end{enumerate}

In phase~\ref{part1}, we essentially make use of the LP objective.
To prove that it works successfully, we will show the following two lemmas.

\begin{lemma}
\label{lemma::firstStepOptLP}
Using the fitness function $f_2$, the expected number of iterations of \algGlobalSemoAlt where the population $\P$ contains a search point $X \in \{0,1\}^n$ such that $\lpPrim(G)=\lpPrim(G[u(X)]) + |X_1|$ and there is an optimal fractional $W$-separator $\{y_v \in \R_{\geq 0}\}_{v \in u(X)}$ of $G[u(X)]$ with $y_v<1$ for every $v \in u(X)$ is upper bounded by $\O(n^3(\log n + \opt))$.
Moreover, once $\P$ contains such a search point at any iteration, the same holds for all future iterations.
\end{lemma} 

\begin{lemma}
\label{lemma::ingredientParateo1}
Let $(A_1,B_1), \dots, (A_m,B_m)$ be a sequence of minimal strictly reducible pairs in $G$, such that $G - \bigcup_{i=1} A_i$ contains no minimal strictly reducible pair.
Let $X \in \{0,1\}^n$ be a sample, such that there is a an optimal fractional $W$-separator $\{y_v \in \R_{\geq 0}\}_{v \in u(X)}$ of $G[u(X)]$ with $y_v<1$ for each $v \in u(X)$.
If $|X_1| + \lpPrim(G[u(X)]) = \lpPrim(G)$, then $A_i \subseteq X_1$ and $B_i \cap X_1 = \varnothing$ for all $i \in [m]$.
\end{lemma}

We guide the rest of this section by using $X \in \{0,1\}^n$ to denote a search point and $(A_1,B_1), \dots, (A_m,B_m)$ as a sequence of minimal strictly reducible pairs,
where $\aall := \bigcup_{i=1}^m A_i$ and $\ball := \bigcup_{i=1}^m B_i$.
For the Algorithm \algGlobalSemoAlt it is unlikely to jump from a uniformly random search point immediately to a search point satisfying \cref{lemma::ingredientParateo1}. 
To guarantee a stepwise progress,
we want that under the condition $\lpPrim(G)=\lpPrim(G[u(X)]) + |X_1|$ at each time $\aall \subsetneq X_1$ there exists a vertex of $v \in \aall \setminus X_1$ in an optimal fractional $W$-separator of $G[u(X)]$ which must have value one.
For this purpose, the characterization of minimal strictly reducible pairs by optimal fractional $W$-separators is useful.

\begin{lemma}[\cite{fomin2019kernelization}, Corollary~6.19 and Lemma~6.20]
\label{lemma::onesInLPReduciblePair}
Let $G=(V,E)$ be an instance of the $W$-separator problem and let $\{y_v \in \R_{\geq 0}\}_{v \in V}$ be an optimal fractional $W$-separator of $G$.
If $G$ contains a minimal strictly reducible pair $(A,B)$, then $y_v=1$ for all $v \in A$ and $y_u = 0$ for all $u \in B$. 	
\end{lemma}

From \cref{lemma::onesInLPReduciblePair} we can derive that if $(\aall \cup \ball) \cap X_1 = \varnothing$, such a vertex $v$ must exist, but the question is what happens if the intersection is not empty.
In particular, we want to avoid vertices of $\ball$ being in $X_1$, since reducible pairs in $G$ may then no longer exist in $G[u(X)]$.
We start with the proof of \cref{lemma::firstStepOptLP} and show later how it is related to a sequence of minimal strictly reducible pairs.
The first lemma is a simple but useful observation.

\begin{lemma}
\label{lemma::lowerBoundLP}
For every $X \in \{0,1\}^n$ it holds that $\lpPrim(G) \leq |X_1| + \lp(G[u(X)])$.
\end{lemma}

If \cref{lemma::lowerBoundLP} is true, it is not difficult to derive that if we have that it holds with equality for a search point $X$, then $f_2(X)$ is a pareto optimal vector of the fitness function $f_2$, as given below as a corollary.

\begin{corollary}
\label{corollary::paretoOptLP}
If a search point $X \in \{0,1\}^n$ satisfy $|X_1| + \lp(G[u(X)]) = \lpPrim(G)$, then the vector $(|X_1|,*,\lp(G[u(X)])$ is a pareto optimal vector of the fitness function $f_2$.
\end{corollary}

The next lemma ensures that removing vertices with value one in an optimal fractional $W$-separator does not affect the objective of a fractional $W$-separator of the remaining graph.

\begin{lemma}[\cite{fomin2019kernelization}, Corollary 6.17]
\label{lemma::onesInLPSameObjective}
Let $G=(V,E)$ be an instance of the $W$-separator problem and let $\{y_v \in \R_{\geq 0}\}_{v \in V}$ be an optimal fractional $W$-separator of $G$.
Let $V' \subseteq V(G)$, such that $y_v = 1$ for all $v \in V'$.
Then, $\{y_v \mid v \in V \setminus V'\}$ is an optimal fractional $W$-separator of $G-V'$, i.e., ~$\sum_{v \in V \setminus V'} y_v = \lpPrim(G-V')$.
\end{lemma}

\cref{corollary::paretoOptLP,lemma::onesInLPSameObjective} allow incremental progress in the set of 1-bits with respect to search points $X \in \P$ that satisfy $|X_1| + \lp(G[u(X)]) = \lpPrim(G)$ without backstepping.
With this ingredient we can prove \cref{lemma::firstStepOptLP} (see~Section~\ref{appendix::f2} for a proof).
Since $f_3$ has one less objective than $f_2$, one can derive the following lemma.
\begin{lemma}
\label{corollary::firstStepOptLP}
Using the fitness function $f_3$, the expected number of iterations of \algGlobalSemoAlt where the population $\P$ contains no search point $X \in \{0,1\}^n$ such that $\lpPrim(G)=\lpPrim(G[u(X)]) + |X_1|$ and there is an optimal fractional $W$-separator $\{y_v \in \R_{\geq 0}\}_{v \in u(X)}$ of $G[u(X)]$ with $y_v<1$ for every $v \in u(X)$ is upper bounded by $\O(n^2(\log n + \opt))$.
\end{lemma}

Our next goal is to prove \cref{lemma::ingredientParateo1}.
To identify the head vertices $\aall$ with respect to an optimal fractional $W$-separator, we want to ensure that the selection of the vertices of $\ball$ are distinguishable so that it cannot come to a conflict with \cref{lemma::firstStepOptLP}.
To do this, we will make use of the LP-objective and show that for a search point $X$ with $X_1 \cap \ball \neq \varnothing$ we have $\lpPrim(G) < \lpPrim(G[u(X)]) + |X_1|$.
Let $(A,B)$ be a minimal strictly reducible pair in $G$.
The essential idea is to use $(W+1)$-packings in $G[A \cup B]$, since they provide lower bounds for $W$-separators.
From \cref{lemma::packingReduciblePair} one can deduce that $G[A \cup B]$ contains a maximum ($W+1$)-packing $\Q$ of size $|A|$, since every vertex of $A$ is contained exactly in one element of $\Q$.
Inspired by ideas on how to find crown decompositions in weighted bipartite graphs from~\cite{DBLP:conf/esa/Casel0INZ21,DBLP:conf/iwpec/KumarL16}, we prove that removing vertices from $B$ only partially affects the size of the $(W+1)$-packing in $G[A \cup B]$, as stated in the following lemma.

\begin{lemma}
\label{lemma::PackingSizeLargeEnough}
Let $(A,B)$ be a minimal strictly reducible pair in $G=(V,E)$ and let $S \subset A \cup B$ with $|S| \leq |A|$.
If $S \cap B \neq \varnothing$, then $G[(A \cup B)] - S$ contains a packing of size $|A| - |S| + 1$.
\end{lemma}

In contrast, note that removing vertices $S \subseteq A$ from $G[A \cup B]$ would decrease the size of a $(W+1)$-packing by $|S|$,
i.e., a maximum $(W+1)$-packing in $G[A \cup B]-S$ has size $|A|-|S|$.
We moved the proof of \cref{lemma::PackingSizeLargeEnough} to the appendix (Section~\ref{appendix::f2}), since it is more technical and too long given the space constraints.
Essentially, we make use of the following two lemmas and properties of network flows.
In particular, these lemmas describe the new properties we have found for minimal strictly reducible pairs and may be of independent interest.

\begin{lemma}
	\label{lemma::decideWhichA}
	Let $(A,B)$ be a minimal strictly reducible pair in $G$ with parameter $W$.
	Then, for every $a^* \in A$ there is an assignment function $g \colon \B \times A \to \N_0$ like in \cref{def::RedPair} that satisfies $\sum_{C \in \B} g(C,a^*) \geq 2W$ and $\sum_{C \in \B} g(C,a) \geq 2W -1$ for every $a \in A \setminus \{a^*\}$.
\end{lemma}

Concerning \cref{lemma::decideWhichA}, we remark that the new feature to before is that the particular vertex (in the lemma $a^*$) can be chosen arbitrarily.

\begin{lemma}
	\label{lemma::SizeNeoghborhood}
	Let $(A,B)$ be a minimal strictly reducible pair in $G$ with parameter $W$.
	Then, for every $A' \subseteq A$ we have $|V(\B_{A'})| \geq |A'|(2W-1)+1$.
\end{lemma}

To conclude the Phase~\ref{part1} we need to prove \cref{lemma::ingredientParateo1}.
Equipped with \cref{lemma::PackingSizeLargeEnough} we may prove statements about the LP-objective if $X_1 \cap \ball \neq \varnothing$.
In doing so, we prove another relation with respect to such a sequence, which fits the proof and will be useful later.

\begin{lemma}
\label{lemma::RedPairLPNoSequence}
Let $(A_1,B_1), \dots, (A_m,B_m)$ be a sequence of minimal strictly reducible pairs and let $X_1 \in \{0,1\}^n$ be a sample.
\begin{enumerate}
	\item If $X_1 = \bigcup_{i=1}^m A_i$, then $\lpPrim(G) = \lpPrim(G[u(X)]) + |X_1|$.
	\label{enum::1}
	\item If $X_1 \cap B_\ell \neq \varnothing$ for an $\ell \in [m]$, then $\lpPrim(G(u[X])) + |X_1| > \lpPrim(G)$.
	\label{enum::2}
\end{enumerate}
\end{lemma}

Suitable for \cref{lemma::firstStepOptLP} we have characterized the case $X_1 \cap \ball \neq \varnothing$.
It remains to give a relation to this lemma when $\aall \cap X_1 \neq \varnothing$ and $\aall \subsetneq X_1$.
In particular, we want to ensure that in this case at least one vertex of $\aall \setminus X_1$ must be one in an optimal fractional $W$-separator of $G[u(X)]$.

\begin{lemma}
\label{lemma::CrownStaysCrown}
Let $(A,B)$ be a minimal strictly reducible pair in $G$ and let $\hat{A} \subset A$.
Then, there is a partition $A_1, \dots, A_m$ of $A \setminus \hat{A}$ with disjoint vertex sets $B_1, \dots, B_m \subseteq B$, such that for each $i \in [k]$ the tuple $(A_i,B_i)$ is a minimal strictly reducible pair in $G - \hat{A}$.  
\end{lemma}

By \cref{lemma::onesInLPReduciblePair} we already know that the head vertices of a minimal strictly reducible pair in an optimal fractional $W$-separator have value one.
\cref{lemma::CrownStaysCrown} ensures that if some of the head vertices are removed, the value of the remaining head vertices in the respective optimal fractional solution remain one.
The proof of \cref{lemma::ingredientParateo1} can be found in the appendix (Section~\ref{appendix::f2}) and concludes Phase~\ref{part1}.

Next, we prove that Phase~\ref{part2} works successfully.
After Phase~\ref{part1}, we have a search point $X$ in the population $\P$ with $\aall \subseteq X_1$ such that $\lpPrim(G) = \lpPrim(G[u(X)]) + |X_1|$.
Consequently, $|X_1| \leq \opt$ and therefore we can prove that \algGlobalSemoAlt reaches a search point $X'$ with $X'_1 = \aall$ from $X$ in FPT-time.

\begin{lemma}
\label{lemma::XReducedInPop}
Let $G=(V,E)$ be an instance of the $W$-separator problem, and let $(A_1,B_1)$, $\dots$, $(A_m,B_m)$ be a sequence of minimal strictly reducible pairs in $G$, such that $G - \bigcup_{i=1} A_i$ contains no strictly reducible pair.
Using the fitness function $f_2$, the expected number of iterations of \algGlobalSemoAlt until the population $\P$ contains a search point $X$ with $X_1 = \bigcup_{i=1}^m A_i$ is upper bounded by $\O\left(n^3(\log n + \opt) + 2^{\opt}\right)$.
\end{lemma}

The question that remains is whether we keep $X'$ in the population once we find it.
This is where the uncovered-objective and the structural properties of minimal strictly reducible pairs come into play.

\def\xdiff{X_{\texttt{dif}}}
\def\vdiff{V_{\texttt{dif}}}

\begin{lemma}
\label{lemma::paretoOptiSol}
Let $X \in \{0,1\}^n$ and let $(A_1,B_1), \dots, (A_m,B_m)$ be a sequence of minimal strictly reducible pairs in $G$, such that $G - \bigcup_{i=1} A_i$ contains no strictly reducible pair.
If $X_1 = \bigcup_{i=1}^m A_i$, then $X$ is a pareto optimal solution.
\end{lemma}
\begin{proof}
Let $A = \bigcup_{i=1}^m A_i$ and $B = \bigcup_{i=1}^m B_i$.
We will prove that if there is a search point $X'$ that dominates $X$, then $G - A$ contains a minimal strictly reducible contradicting the precondition of the lemma.
Note that a minimal strictly reducible pair in $G - A$ have to be in $G[u(X)]$ as the other components in $G-A$ have size at most $W$.

If $X_1 = \bigcup_{i=1}^m A_i$, then by \cref{lemma::RedPairLPNoSequence} $\lpPrim(G) = \lpPrim(G[u(X)]) + |X_1|$.
That is, we might restrict to solutions $X' \in \{0,1\}^n$ with $\lpPrim(G) = \lpPrim(G[u(X')]) + |X'_1|$ as well as $|X'_1| = |X_1|$ and can focus on the objective $u(X)$ and $u(X')$, respectively.
Note that it cannot happen that $|X'_1| < |X_1|$ and $\lpPrim(G[u(X')]) \leq \lpPrim(G[u(X)])$ as $\lpPrim(G) = \lpPrim(G[u(X)]) + |X_1| > \lpPrim(G[u(X')]) + |X'_1|$ contradicting \cref{lemma::lowerBoundLP}.

W.l.o.g.~we can assume that every connected component of $G-X_1=G-A$ of size at most $W$ is also in $G[B]$, i.e., $V \setminus u(X) = A \cup B$.
Let $g_1 \colon \B_1 \times A_1 \to \N_0,\dots,g_m \colon \B_m \times A_m \to \N_0$ be the according assignment functions of $(A_1,B_1), \dots, (A_m,B_m)$ and let $g \colon \B \times A \to \N_0$ be defined as $g(a,C) = g_i(a,C)$ if $a \in A_i$, $C \in \B_i$ and otherwise $g(a,C) = 0$.
Suppose there is such an $X' \neq X$ as described above with $|u(X')| \leq |u(X)|$, or equivalently $|V \setminus u(X')| \geq |V \setminus u(X)|$.
Note that $|V \setminus u(X)| - |A| \geq |A|(2W-1)+1$, 
since for each $a \in A$ and at least for one $a' \in A$ we have $\sum_{C \in \B} g(C,a) \geq 2W-1$ and $\sum_{C \in \B} g(C,a') \geq 2W$.

We define $V(\B_{\Tilde{A}}) \subseteq B$ for $\Tilde{A} \subseteq A$ as the vertices in the components $\{C \in \B \mid N(C) \cap \Tilde{A} \neq \varnothing\}$. 
Let $\xdiff = X'_1 \setminus  X_1 = X'_1 \setminus A$ and let $\vdiff = (V \setminus u(X')) \setminus (V \setminus u(X))$.
Note that $\vdiff \subseteq u(X)$ and $|\xdiff| = |A \setminus X'_1|$ by $|A| = |X'_1|$.
From $\lpPrim(G) = \lpPrim(G[u(X')]) + |X'_1|$, we obtain by \cref{lemma::ingredientParateo1} that $\xdiff \cap B = \varnothing$ and therefore $\xdiff \subseteq \vdiff$ as no vertex of $\xdiff$ is in $A \cup B = V \setminus u(X)$.
Thus, by the assignment function $g$ each vertex in $A' = A \setminus X'_1$ is in a connected component of size at least $W+1$ of $G[u(X')]$ and therefore $A'$ as well as $V(\B_{A'})$ are not in $V \setminus u(X')$.
Furthermore, for at least one $j \in [m]$ we have $A_j \subsetneq X'_1$ and by \cref{lemma::SizeNeoghborhood} we obtain that $|V(\B_{A_j \setminus X'_1})| \geq |A_j \setminus X'_1| (2W-1)+1$.
That is, to satisfy now $|V \setminus u(X')| \geq |V \setminus u(X)|$ we must have $|\vdiff| - |\xdiff| \geq |\xdiff|(2W-1)+1$ as at least $|A \setminus X'_1|(2W-1)+1+|A \setminus X'_1|$ vertices are in $V \setminus u(X)$ that are not in $V \setminus u(X')$.
Observe that each connected component $C$ of $G[\vdiff - \xdiff]$ satisfy $N(C) \cap \xdiff \neq \varnothing$ and $|C| \leq W$.
From this, we obtain that $G[\vdiff]$ contains a strictly reducible pair $(\hat{A},\hat{B})$ with $\hat{A} \subseteq \xdiff$ and $\hat{B} \subseteq \vdiff \setminus \xdiff$ by \cref{lemma::balancedExpansionStrictlyReducibleExist} and therefore also a minimal strictly reducible pair $(\hat{A}',\hat{B}')$ with $\hat{A}' \subseteq \hat{A}$ and $\hat{B}' \subseteq \hat{B}$.
In particular, $(\hat{A}',\hat{B}')$ is a minimal strictly reducible in $G-A$, since $\vdiff \subseteq u(X)$ and the connected components in $G[\vdiff \setminus \xdiff]$ exist identically in $G[u(X)]-\xdiff$.
\end{proof}
We are ready for the final theorem of this section, which shows that Phase~\ref{part3} also works successfully.
\paragraph{Proof of \cref{thm::OptSolf2}:}
Let $(A_1,B_1),\dots,(A_m,B_m)$ be a sequence of minimal strictly reducible pairs, such that $G-\bigcup_{i=1}^m A_i$ contains no strictly reducible pair.
Furthermore, let $\P$ be a population with respect to $f_2$ in the algorithm \algGlobalSemoAlt. 
By \cref{lemma::XReducedInPop} we have a search point $X \in \P$ with $X_1 = \bigcup_{i=1}^m A_i$ after $\O\left(n^3(\log n + \opt) + 2^{\opt}\right)$ iterations  in expectation. 
Moreover, $X$ is a pareto optimal solution by \cref{lemma::paretoOptiSol}.

Since $G[u(X)]$ contains no strictly reducible pair, we can derive from \cref{lemma::existenceReducible} that $|V(G[u(X)])| \leq 2 \cdot \opt \cdot W$. 
The algorithm \algGlobalSemoAlt calls with $1/3$ probability the mutation that flips every vertex $u(X)$ with $1/2$ probability in $X$.
That is, reaching a state $X'$ from $X$, such that $X'_1 = V^*$ has a probability of at least $\Omega\left(2^{-2 \cdot \opt \cdot W}\right)$, where  selecting $X'$ in $\P$ has probability $\Omega(1/n^2)$ (cf.~\cref{lemma::singleFlipAndBoundedPop}).
Thus, once $X$ is contained in $\P$ it takes in expectation $\O\left(n^2 \cdot 4^{\opt \cdot W} \right)$ iterations reaching $X'$.
As a result, the algorithm needs in total $\O\left(n^3(\log n + \opt) + n^2 \cdot 4^{\opt \cdot W} \right)$ iterations finding an optimal $W$-separator in expectation.

%% file: g_apx.tex

In this section we consider the $W$-separator problem with the fitness $f_2$ and $f_3$ associated with \algGlobalSemo and \algGlobalSemoAlt.
We show that the algorithms find approximate solutions when we reduce their overhead.
In particular, we prove the following theorems.

\begin{theorem}
	\label{thm::approx}
	Using the the fitness function $f_3$, the expected number of iterations of \algGlobalSemo until it finds a $(W+1)$-approximation in $G=(V,E)$ is upper bounded by $\O\left(n^2(\log n + W \cdot \opt)\right)$.
\end{theorem}

\begin{theorem}
	\label{thm::ptas}
	Let $G=(V,E)$ be an instance of the $W$-separator problem and let $\varepsilon \in [0,1)$.
	
	\begin{enumerate}
		\item\label{thm::item1APX}
		Using the fitness function $f_2$, the expected number of iterations of \algGlobalSemoAlt until an $(1 + \varepsilon(3/2 W - 1/2))$-approximation is sampled is upper bounded by\\ $\O\left(n^3(\log n + W \cdot \opt) + 2^\opt + n^2 \cdot 4^{(1-\varepsilon) \opt \cdot W} \right)$.
		\item\label{thm::item2APX} 
		Using the fitness function $f_3$, the expected number of iterations of \algGlobalSemoAlt until a $(2 + \varepsilon(3/2 W - 1/2))$-approximation is sampled is upper bounded by\\ $\O\left(n^2(\log n + W \cdot \opt) + n \cdot 4^{(1-\varepsilon) \opt \cdot W}\right)$.
	\end{enumerate}
\end{theorem}
 
Note that \cref{thm::ptas} implies that we can hope for incremental progress towards an optimal solution if we compare it to \cref{thm::OptSolf2}.
Note also that \cref{thm::ptas}~(\ref{thm::item1APX}) has a running time of $\O\left(n^3(\log n + W \cdot \opt) + n^2 \cdot 4^{(1-\varepsilon) \opt \cdot W} \right)$ if $\varepsilon < 1/2$.

To prove our theorems, we show that once there is a search point in the population that has a desired target value with respect to the LP-objective and the one-objective, then the algorithms find in polynomial time a $W$-separator that does not exceed this target value.
That is, the 1-bits of this search point do not necessarily have to form a $W$-separator.

\begin{lemma}
	\label{lemma::apxIsSave}
	Let $G=(V,E)$ be an instance of the $W$-separator problem, $\P$ a population with respect to the fitness function $f_2$ or $f_3$, $c > \opt$, and $X \in \P$ a search point satisfying $|X_1| + (W+1) \cdot \lpPrim(G[u(X)]) \leq c$.
	Using the the fitness function $f_2$ or $f_3$, the expected number of iterations of \algGlobalSemo until it finds a $W$-separator $S$ in $G$ with $|S| \leq c$ is upper bounded by $\O\left(n^2 W \cdot \opt\right)$ or $\O\left(n^3 W \cdot \opt\right)$, respectively.
\end{lemma}

We conclude this section with the prove of \cref{thm::ptas}~(\ref{thm::item1APX}).
Thereby, we basically need to show that we reach in the stated runtime a search point $X$ that satisfies the precondition of \cref{lemma::apxIsSave} with the desired approximation value.

\def\Xred{X_{\texttt{red}}}

\paragraph{\textbf{Proof of \cref{thm::ptas}~(\ref{thm::item1APX})}:}
Let $X \in \{0,1\}^n$ be a search point such that $G[u(X)]$ contains no minimal strictly reducible pair (\emph{irreducible-condition}).
Furthermore, let $S$ be an optimal solution of $G[u(X)]$ and let $U = u(X) \setminus S$.
Note that $|S| \leq \opt$.
Since $G[u(X)]$ contains no minimal strictly reducible pair, we have $|U| = |u(X)| - |S| \leq 2W|S| - |S| = |S|(2W-1)$ by \cref{lemma::existenceReducible}.

Recall that \algGlobalSemoAlt chooses with $1/3$ probability the mutation that flips every bit corresponding to the vertices in $u(X)$ with $1/2$ probability.
From this, the search point $X$ has a probability of $\Omega\left(2^{-(1-\varepsilon) |S| - (1-\varepsilon) |S| \cdot (2W-1) }\right) = \Omega\left(4^{-(1-\varepsilon) |S| \cdot W }\right)$ to flip $(1-\varepsilon) |S|$ fixed vertices of $S$ and to not flip $(1-\varepsilon) |S| \cdot (2W-1)$ fixed vertices of $U$ in one iteration.
Independently from this, half of the remaining vertices of $S$ and $U$ are additionally flipped in this iteration, i.e., $\frac{1}{2} \varepsilon |S|$ of $S$ and  $\frac{1}{2} \varepsilon |U|$ of $U$.
Let $S'$ and $U'$ be the flipped vertices in this iteration and let $X'$ be the according search point.
Note that $\lpPrim(X') \leq |S| - |S'|$ simply because there is a $W$-separator of $G[u(X)] - S'$ of size $|S| - |S'|$.
Hence, we have
\begin{align*}
	|X'_1| + (W+1) \cdot \lpPrim(G[u(X')])
	& = |X_1| + |S'| + |U'| + (W+1) \cdot \lpPrim(G[u(X')])\\
	&\leq |X_1| + |S'| + |U'| +  (W+1) \cdot (|S| - |S'|)\\
	&= |X_1| + |S|(W+1) - |S'|W + |U'|.
\end{align*}
Next, we upper bound $|S'|$ and $|U'|$ in terms of $|S| \leq \opt$.
Using the fact $|U| \leq |S| (2W-1)$, we obtain 
$|U'| = \frac{1}{2} \varepsilon |U| \leq \frac{1}{2} \varepsilon |S|(2W-1) = \varepsilon |S| W - \frac{1}{2} \varepsilon |S|$.
Regarding $S'$ we have
$|S'| = (1-\varepsilon)|S| + \frac{1}{2} \varepsilon |S| = |S| - \varepsilon |S| + \frac{1}{2} \varepsilon |S|$.
As a result, we obtain
\begin{align*}
	& |X'_1| + (W+1) \cdot \lpPrim(G[u(X')])\\
	& \leq |X_1| + |S|(W+1) - |S'|W + |U'|\\
	&\leq |X_1| + |S|(W+1) - (|S| - \varepsilon |S| + \frac{1}{2} \varepsilon |S|)W + \varepsilon |S| W - \frac{1}{2} \varepsilon |S|\\
	&\leq |X_1| + |S|(W+1) - |S|W + \varepsilon |S| W - \frac{1}{2} \varepsilon |S| W + \varepsilon |S| W - \frac{1}{2} \varepsilon |S|\\
	&= |X_1| + |S| + |S|\left(2 \varepsilon W - \frac{1}{2} \varepsilon W  - \frac{1}{2} \varepsilon\right). 
\end{align*}
Observe that once a desired $X$ is guaranteed to be in the population, an event described above occurs after $\O\left(n^2 \cdot 4^{(1-\varepsilon) |S| \cdot W }\right)$ iterations in expectation for the fitness functions $f_2$, where the factor $n^2$ comes from selecting $X$ (cf.~\cref{lemma::singleFlipAndBoundedPop}).
	
	Let $(A_1,B_1),\dots,(A_m,B_m)$ be a sequence of minimal strictly reducible pairs, such that $G-\bigcup_{i=1}^m A_i$ contains no strictly reducible pair.
	By \cref{lemma::XReducedInPop} we have a search point $X$ in the population $\P$ with $X_1 = \bigcup_{i=1}^m A_i$ after $\O\left(n^3(\log n + \opt) + 2^{\opt}\right)$ iterations in expectation.
	Note that $X$ satisfies the irreducible-condition.
	Moreover, $X$ is a pareto optimal solution by \cref{lemma::paretoOptiSol}.
	By \cref{thm::saveReduction} we have $|X_1| = \opt - |S|$.
	Using that $|S| \leq \opt$, we obtain
	\begin{align*}
		|X'_1| + (W+1) \cdot \lpPrim(G[u(X')])
		&\leq |X_1| + |S| + |S|\left(2 \varepsilon W - \frac{1}{2} \varepsilon W  - \frac{1}{2} \varepsilon\right)\\
		&= \opt - |S| + |S| + |S|\left(2 \varepsilon W - \frac{1}{2} \varepsilon W  - \frac{1}{2} \varepsilon\right)\\
		&\leq \opt\left(1 + \varepsilon\left(\frac{3}{2}W - \frac{1}{2}\right)\right).
	\end{align*}
	As a result, by the choice of $X$ the resulting search point $X'$ satisfies the precondition of \cref{lemma::apxIsSave} with $c= \opt\left(1 + \varepsilon\left(\frac{3}{2}W - \frac{1}{2}\right)\right)$.
	That is, once $X'$ is in the population $\P$, the algorithm \algGlobalSemoAlt need in expectation $\O(n^3 W \cdot \opt)$ iterations having a search point in $\P$ which is a $\left(1 + \varepsilon\left(\frac{3}{2}W - \frac{1}{2}\right)\right)$-approximation. 
	In summary, in expectation the desired search point $X'$ is in $\P$ after $\O\left(n^3(\log n + W \cdot \opt) + 2^\opt + n^2 \cdot 4^{(1-\varepsilon) \opt \cdot W} \right)$ iterations.

%% file: appendix_prelim.tex

\subsection*{Proof of \cref{lemma::singleFlipAndBoundedPop}}
\begin{proof}
	We start to bound the size of the population $\P$.
	Given $f_1$ or $f_2$ for the bit sequences $\P$, the first and second entries each have at most $(n+1)$ distinct values, and we keep at most one for each possibility.
	Therefore, we have $|\P| \leq (n+1)^2$ and thus selecting a certain search point $X \in \P$ has probability $\Omega(1/n^2)$.
	In \algGlobalSemo we mutate every bit in $X$ with probability $1/n$.
	That is, we obtain a probability of $1/n \cdot (1 - 1/n)^{n-1} \geq 1/ne \in \Omega(1/n)$ to flip a certain bit in $X$.
	This probability decreases only by a constant factor of $1/3$ in \algGlobalSemoAlt and keeps therefore the probability by $\Omega(1/n)$ for this event.
	As a result, selecting a certain search point $X \in \P$ and flipping only one single bit in it has probability $\Omega(1/n^3)$ for both algorithms.
\end{proof}

\subsection*{Proof of \cref{lemma::zeroSol}}

\begin{proof}
	Let $\P \neq \varnothing$ and let $X^{\min} \in \P$ be the bit string in $\P$ with the minimal number of ones, where $i := |X^{\min}_1|$.
	By \cref{lemma::singleFlipAndBoundedPop} the probability that $X^{\min}$ is chosen in one iteration of both algorithms is $\Omega(1/n^2)$.
	If we consider \algGlobalSemo, then the probability that the mutation of $X^{\min}_1$ results in a bit string $X$ with $|X_1| < i$ is at least 
	$i/n \cdot (1 - 1/n)^{n-1} \geq i/ne$, which is the probability that only one 1-bit is flipped and nothing else.
	Note that the probability changes by a factor of $1/3$ in the algorithm \algGlobalSemoAlt.
	Thus, the probability for both algorithms that the population gets after an iteration a bit string that has less number of ones compared to the previous population is $\Omega(i/n^3)$.
	This in turn means that the expected number of iterations this happens is $\O(n^3/i)$.  
	Using the method of fitness based partitions \cite{DBLP:conf/ppsn/Sudholt10} and summing up over the different values of $i$ we obtain an expected time of $\sum_{i=1}^{n} \O(n^3/i) = \O(n^3 \log n)$ that $0^n$ is in $\P$.
\end{proof}

%% file: appendix_f1.tex

\subsection*{Proof of \cref{lemma::reducedInstance_f1}}
\begin{proof}
	Let $V' = \{v_1, \dots, v_\ell\} \subseteq V$ be the vertices with degree larger than $k+W$ (\emph{reducible vertices}), such that $d(v_i) \geq d(v_j)$ for $i>j$.
	Observe that if $V' = \varnothing$, then $0^n$ is already the desired search point and we are done by \cref{lemma::zeroSol}.
	For $i \in \{0,1,\dots,\ell\}$ let $X^i$ be a search point with $|X^i_1| = i$ and $\sum_{v \in X^i_1} d(v) = \sum_{j=1}^i d(v_j)$.
	In particular, $X^\ell$ is the desired search point according to the lemma.
	First observe that $X^i$ can only be dominated by a search point $X$ if $|X_1| \leq |X^i_1|$ and $-\sum_{v \in X_1} d(v) \leq - \sum_{v \in X^i_1} d(v)$,
	which is only possible if $|X_1| = i = |X^i_1|$ and $-\sum_{v \in X} d(v) = - \sum_{v \in X'_1} d(v)$ as $X^i_1$ contains only the vertices of largest degree.
	Consequently, $X^i \subseteq V'$ and once $X^i$ is in the population $\P$ the vector $(|X^i_1|, *, -\sum_{v \in X_1} d(v))$ is pareto optimal.
	
	Let $\P$ be a population with $0^n \in \P$ and let $s < \ell$ be the largest integer such that $X^s \in \P$.
	Note that $X^0 = 0^n$.
	Let $v \in u(X^s)$ be a vertex that satisfies $d(v) = \max_{u \in u(X^s)} d(u)$.
	By \cref{lemma::singleFlipAndBoundedPop} \algGlobalSemoAlt flips only a certain bit from a certain search point of $\P$ with probability $\Omega(1/n^3)$ and
	thus, mutating $X^s$ to a search point $X^{s+1}$ takes in expectation $\O(n^3)$ iterations.
	Using the method of fitness based partitions \cite{DBLP:conf/ppsn/Sudholt10} and summing up over the different values of $s$ leads to $\sum_{i=1}^{\ell} \O(n^3) \leq \sum_{i=1}^{\opt} \O(n^3) = \O(\opt \cdot n^3)$ expected number of iterations having $X^\ell$ in the population $\P$, once $0^n \in \P$.
	By \cref{lemma::zeroSol} the expected number of iterations such that $0^n$ is in the population $\P$ is $\O(n^3 \log n)$.
	As a result, we have $X^\ell \in \P$ after $\O(n^3(\opt + \log n))$ iterations in expectation. 
\end{proof}

\subsection*{Proof of \cref{thm::fitness1Opt}:}
\begin{proof}
First observe that once we have a search point $X$ according to \cref{lemma::reducedInstance_f1} in the population that this is pareto optimal, since there is no other search point with same or less number of selected vertices that yields to a smaller value of $-\sum_{v \in X_1} d(v)$.
On the other hand, by \cref{thm::degreeKernel} we know that the uncovered-objective $u(X)$ has size at most $q=\opt \cdot W(\opt+W)+\opt$
while $X_1$ contains only vertices that have to be in the optimal solution.
In expectation, $X$ is in the population after $\O(n^3(\opt + \log n))$ iterations of \algGlobalSemoAlt (cf.~\cref{lemma::reducedInstance_f1}).
Thus, given $X$ is selected in the Algorithm~\algGlobalSemoAlt we obtain a probability of $1/3 \cdot 2^{-q}$ flipping $m \leq \opt$ vertices of $u(X)$ that lead to an optimal solution. 
That is, if $X \in \P$ we have a probability of $\Omega\left(1/n^2 \cdot 2^{-q}\right)$ reaching the optimal solution in one iteration, where the additional factor $1/n^2$ comes from selecting $X$ (cf.~\cref{lemma::singleFlipAndBoundedPop}).
Consequently, the expected number of iterations reaching an optimal solution is upper bounded by $\O\left(n^3(\opt + \log n) + n^2 \cdot 2^q \right)$.
\end{proof}

%% file: appendix_f2.tex

\subsection*{Proof of \cref{lemma::balancedExpansionStrictlyReducibleExist}}
For the proof we use a structural lemma provided by Casel et al.~\cite{DBLP:conf/esa/Casel0INZ21} for potential vertices of a strictly reducible pair with additional properties for the unincluded vertices.
For a better understanding, it is convenient to think by an assignment function $g \colon \B \times A \to \N_0$ on flows, which sends fractional vertex-parts of components $\B$ to $A$.
For $g$ we define $\overrightarrow{\B}_a :=\left\{C \in \B \mid g\left(C,a\right) > 0 \right\}$ for $a \in A$, $\overrightarrow{\B}_{A'} := \bigcup_{a \in A'} \overrightarrow{\B}_a$ for $A' \subseteq A$ and $\overrightarrow{\B}_U :=\left\{C \in \B \mid \sum_{a \in A} g\left(C,a\right) < |C| \right\}$.

\begin{definition}[fractional balanced expansion]
	\label{definition::fbe}
	Let $G=\left(A \cup B, E\right)$ be a graph, where for each component $C \in \B$ we have $|C| \leq W$. 
	For $q \in \N_0$, a partition $A_1\cup A_2$ of $A$ and $g \colon \B \times A \to \N_0$, the tuple $\left(A_1,A_2,g,q\right)$ is called \emph{fractional balanced expansion} if:
	\begin{enumerate}
		\item for all $C \in \B$ and $a \in A$, if $g(C,a) \ne 0$, then $a \in N(C)$,
		\item $\sum_{C \in \B} g\left(C,a\right)\, \begin{cases}\geq \ q, & a\in A_1\\ 
			\leq \ q, & a\in A_2\end{cases}$  
		\item $\forall C \in \B \colon \sum_{a \in A} g\left(C,a\right) \leq |C|$ 
		\item  $N\left(\overrightarrow{\B_U} \cup \overrightarrow{\B}_{A_1}\right) \subseteq A_1$
	\end{enumerate}
\end{definition}

The vertices $A_1$ in a fractional balanced extension correspond to the head vertices, and the question we are interested in is when they are nonempty. 

\begin{lemma}[\cite{DBLP:conf/esa/Casel0INZ21}, Lemma~4]
	\label{lemma::balancedExpansionReducibleExist}
	Let $G=\left(A \cup B, E\right)$ be a graph, where for each $C \in \B$ we have $|C| \leq W$ and no $C \in \B$ is isolated, i.e., $N(C) \neq \varnothing$.
	For every $q \in \N_0$, if $w(B) \geq q|A|$, then there exists a fractional balanced expansion $\left(A_1,A_2,g,q\right)$ with $A_1 \neq \varnothing$. 
\end{lemma}

Equipped with \cref{lemma::balancedExpansionReducibleExist} we can prove \cref{lemma::balancedExpansionStrictlyReducibleExist}, i.e., the existence of a strictly reducible pair with respect to the vertex-size. 

\paragraph{\textbf{\textbf{Proof of \cref{lemma::balancedExpansionStrictlyReducibleExist}:}}}
By the precondition of the lemma we obtain by \cref{lemma::balancedExpansionReducibleExist} the existence of a fractional balanced expansion $\left(A_1,A_2,g,2W-1\right)$ with $A_1 \neq \varnothing$ in $G$.
By the properties of the fractional balanced expansion we have $N\left(\overrightarrow{\B}_U \cup \overrightarrow{\B}_{A_1}\right) \subseteq A_1$
and $\sum_{b \in B} g\left(C,a\right) \geq 2W-1$ for all $a \in A_1$.
That is, $(A'=A_1, B' = \bigcup_{C \in \B_U \cup \overrightarrow{\B}_{A_1} } C)$ is a reducible pair.
Since $\sum_{C \in \B} g\left(C,a\right) \leq 2W-1$ for all $a \in A_2$ and no $C \in \B$ is isolated, we have
\begin{align*}
    |B| - \sum_{a \in A_2} \sum_{C \in \B} g(C,a) &\geq |A|(2W-1) + 1 - |A_2| (2W-1)\\
    &= |A_1|(2W-1)+1.
\end{align*}
Thus, either there is already an $a \in A_1 = A'$ with $\sum_{C \in  \B} g\left(C,a\right) \geq 2W$, or $\overrightarrow{\B_U} \neq \varnothing$ and we can assign one more unit of $\overrightarrow{\B}_U$ to a vertex of $A_1 = A'$ which has a component $C \in \overrightarrow{\B_U}$ in its neighborhood.
This in turn shows that $(A',B')$ is a strictly reducible pair.

\subsection*{Proof of \cref{lemma::firstStepOptLP}}
\begin{proof}
Let $\P$ be a population with $0^n \in \P$.
Let $s \geq 0$ be the largest integer that denotes a sample $X^s$ in $\P$ with $|X^s_1| = s$ and $\lpPrim(G) = \lpPrim(G[u(X^s)]) + |X^s_1|$.
Note that $X^0=0^n$.
By \cref{corollary::paretoOptLP} the tuple $(|X^s_1|,*,\lpPrim(u(X^s)))$ is pareto optima, i.e., $s$ can never decrease.
If there is an optimal fractional $W$-separator $\{y_v \in \R_{\geq 0}\}_{v \in u(X)}$ of $G[u(x)]$, such that $y_v<1$ for every $v \in u(X)$ then we are done.
Otherwise, there is a $v \in u(X)$ with $y_v=1$.
By \cref{lemma::onesInLPSameObjective} we obtain that $\lpPrim(G) = \lpPrim(G[u(X^s) \setminus \{v\}]) + |X^s_1 \cup \{v\}|$.
Thus, mutating $X^s$ by only flipping the according entry $v$ in $X^s$ results in a search point $X^{s+1}$.
By \cref{lemma::singleFlipAndBoundedPop} this event has probability $\Omega(1/n^3)$ in the algorithm \algGlobalSemoAlt, and happens at most $\opt$ times as $\lpPrim(G) \leq \opt$.
That is, $s \leq \opt$ and mutating $X^s$ to a search point $X^{s+1}$ happens after $\O(n^3)$ iterations in expectation.
Moreover, by \cref{lemma::zeroSol} the search point $0^n$ is in the population $\P$ after $\O(n^3 \log n)$ iterations in expectation. 
By using the method of fitness based partitions \cite{DBLP:conf/ppsn/Sudholt10}, the fact that a search point described in the lemma is not in $\P$ is upper bounded by $\O(n^3(\log n + \opt))$ iterations in expectation.
\end{proof}

\subsection*{Proof of \cref{lemma::ingredientParateo1}:}
As $|X_1| + \lpPrim(G[u(X)]) = \lpPrim(G)$ we obtain by \cref{lemma::RedPairLPNoSequence} that $X_1 \cap B_i = \varnothing$ for every $i \in [m]$.
Remains to show that $A_i \subseteq X_1$ for every $i \in [m]$.

Let $(A,B)$ be a strictly reducible pair.
We define $\B_{\Tilde{A}} \subseteq \B$ for $\Tilde{A} \subseteq A$ as the components $\{C \in \B \mid N(C) \cap \Tilde{A} \neq \varnothing\}$. 
Assume $A_j$ is a set with $A_j \subsetneq X_1$ for $j \in [m]$, where $A_i \subseteq X_1$ for $i \in [j-1]$.
Let $g$ be an assignment function like in \cref{def::RedPair} for $(A_j,B_j)$.
Since no vertex of $B_j$ is in $X_1$, the vertices $A_j' = A_j \setminus X_1$ along with $B_j' = \bigcup_{C \in \B_{A_j'}} C$ have to be in $u(X)$ as every $a \in A_j'$ satisfies $\sum_{C \in \B_{A_j'}} g(C,a) \geq 2W-1$.
Note that if $X_1 \cap B_j' = \varnothing$, then $\sum_{C \in \B_{A_j'}} g(C,a) \geq 2W-1$ for every $a \in A_j'$ implies that the every vertex $a$ is in a connected subgraph of size at least $W+1$ in $G-X_1$, which in turn implies that $A_j'$ and $B_{A_j'}$ are in $u(X)$.
Combining the fact $\bigcup_{i=1}^{j-1} A_i \subseteq X_1$ with \cref{lemma::CrownStaysCrown}  we obtain that there is a partition $\hat{A}_1, \dots, \hat{A}_p$ of $A_j'$ with disjoint vertex sets $\hat{B}_1, \dots, \hat{B}_p \subseteq \bigcup_{C \in \B_{A_j'}} C$, such that for each $\ell \in [p]$ the tuple $(\hat{A}_\ell,\hat{B}_\ell)$ is a minimal strictly reducible pair in $G[u(X)]$.
However, by \cref{lemma::onesInLPSameObjective} we have then $y_a=1$ for every $a \in A_j'$ contradicting the precondition of the lemma.

\subsection*{Proof of \cref{lemma::lowerBoundLP}}
\begin{proof}
	The fractional $W$-separator of $G[u(X)]$ with an objective value of $LP(G[u(X)])$ can be extended to a fractional $W$-separator of $G$ by the vertices $X_1$.
	Since this solution would be feasible for the linear program of the $W$-separator problem, we obtain the desired inequality.
\end{proof}

\subsection*{Proof of \cref{lemma::PackingSizeLargeEnough,lemma::decideWhichA,lemma::SizeNeoghborhood}}
To prove \cref{lemma::PackingSizeLargeEnough,lemma::decideWhichA,lemma::SizeNeoghborhood}, we recall some definitions and introduce new ones.
Let $(A,B)$ be a minimal strictly reducible pair with an assignment function $g$ (cf.~\cref{def::RedPair}).
We use $\B$ to denote the connected components of $G[B]$ as vertex sets.
Moreover, we define $\overrightarrow{\B}_a :=\left\{C \in \B \mid g\left(C,a\right) > 0 \right\}$ for $a \in A$, $\overrightarrow{\B}_{A'} := \bigcup_{a \in A'} \overrightarrow{\B}_a$ for $A' \subseteq A$ and $\overrightarrow{\B}_U :=\left\{C \in \B \mid \sum_{a \in A} g\left(C,a\right) < |C| \right\}$.
For a subset $\B' \subseteq \B$ we define $V(\B') := \bigcup_{C \in \B'} C$.
We abuse notation and use $N(\B')$ to denote $N(V(\B'))$.
For $A' \subseteq A$ we define $\B_{A'} := \{C \in \B \mid N(C) \subseteq A'\}$.

For the prove we also make use of directed graphs in the sense of flow networks, where we work with the well-known Ford-Fulkerson framework~\cite{ford1956maximal}.
We denote a network $H$ by $(V,\ora{E},c)$, where $V$ is its set of vertices, $\ora{E}$ is the set of arcs, and $c \colon \ora{E} \to \N$ its capacity function.
The main characteristic of a network is that its vertices contain $s$ and $t$, called the source and sink, respectively.
The source vertex is characterized by having only outgoing arcs, while the sink vertex has only incoming arcs.
We denote a flow by $Z$, i.e., $Z := \left\{z_{\ora{e}} \in \N_{\geq 0} \mid \ora{e} \in \ora{E}\right\}$.
A flow $Z$ is called feasible if $z_{\ora{e}} \leq c(\ora{e})$ for all $\ora{e} \in \ora{E}$ and if $\sum_{\ora{uv} \in  \ora{E}} z_{\ora{uv}} = \sum_{\ora{vw} \in  \ora{E}} z_{\ora{vw}}$ for all $v \in V \setminus \{s,t\}$.
Usually, the challenge is to find a feasible flow that sends as much flow as possible from $s$ to $t$, where the total flow value can be extracted from $\sum_{\ora{ut} \in \ora{E}} z_{\ora{ut}} = \sum_{\ora{sv} \in \ora{E}} z_{\ora{sv}}$.
The residual network $R = (V, \ora{E}')$ is defined with respect to a flow $Z$.
For our purposes, we can simplify its definition in comparison as it is usually known: for $z_{\ora{uv}} \in Z$ we have $\ora{uv} \in \ora{E}'$ if $c(\ora{uv}) - z_{\ora{uv}} > 0$ and $\ora{vu} \in \ora{E}'$ if $z_{\ora{uv}} > 0$.
An $s$-$t$-path in $R$ is called an augmenting path and gives the way how an $s$-$t$-flow with respect to $Z$ can be increased.  
It is a well-known result that if there is no $s$-$t$-path in $R$, then $Z$ is a maximum flow for $s$ and $t$. 

Mainly, we found new properties concerning minimal strictly reducible pairs which are useful for the proof of \cref{lemma::PackingSizeLargeEnough}.
One of them is that the particular vertex in $A$, which is guaranteed to have a larger mapping of fractional vertices from components in $\B$ over $g$, can be chosen arbitrarily.
To prove this, we construct two similar flow networks $H$ and $H_a$ for $a \in A$ with respect to $A, B$ and $g$, which we will call \emph{$g$-embedded flow networks}.
Basically, it describes the assignments from $\B$ to $A$ via $g$ with the advantage that we can make use of the properties provided by a flow network. 
For the proofs it is more convenient to reduce the assignments of each $a \in A$ that are given via $g$ arbitrary from $\oB_{a}$ so that we have $\sum_{C \in \B} g(C,a) = 2W-1$.
Now we describe the procedure getting the $g$-embedded flow network $H = \left(A \cup \B \cup \left\{s,t\right\}, \overrightarrow{E}, c\right)$.
We add the vertices $s,t$, $s$ as source and $t$ as sink, and arcs $\overrightarrow{E}$ with a capacity function $c\colon\overrightarrow{E} \to \mathbb{N}$ defined as follows:
connect every $C \in \B$ through an arc $\overrightarrow{sC}$ with capacity $|C|$ and connect every $a \in A$ through an arc $\overrightarrow{at}$ with capacity $2W-1$. 
Moreover, for every $C \in \B$ and $a \in A$ add an arc $\ora{Ca}$ to $\ora{E}$ if $a \in N(C)$ with capacity $|C|$.
Let $Z = \left\{z_e \in \N_0 \mid \ora{e} \in \ora{E} \right\}$ denote the flow values in $H$.
We embed now $g$ naturally in $H$.
For this, we set $z_{\ora{Ca}} = g(C,a)$ for every $a \in A$ and $C \in \B$.
To ensure now that the flow conversation is ensured for the vertices $V(H) \setminus \{s,t\}$ we set $z_{\ora{sC}} = \sum_{a \in A} g(C,a)$ for each $C \in \B$ and $z_{\ora{at}} = \sum_{C \in \B} g(C,a)$ for each $a \in A$.
It is not difficult to see that no arc capacity is violated and hence we have a feasible flow from $s$ to $t$ in $H$.
We denote by $F = |A|(2W-1)$ the total flow value $\sum_{a \in A} z_{\ora{at}}$.   
Note that we have $z_{\ora{at}} = 2W-1$ for each $a \in A$, which means that these arcs are \emph{saturated} as $c(\ora{at}) = 2W-1$ and therefore $Z$ is a maximum flow in $H$.

For $a \in A$ the flow network $H_a$ is basically the network $H$ with the difference that $c(\ora{at}) = 2W$.
We will accordingly denote the flow in $H_a$ by $Z^a = \left\{z^a_e \in \N_0 \mid \ora{e} \in \ora{E} \right\}$ and its flow value by $F_a$.
Note that this time we do not know, whether we have a maximum flow in hand as $\ora{at}$ is not saturated in $H_a$.

\begin{lemma}
	\label{lemma::flowEverywhere}
	Let $(A,B)$ be a minimal strictly reducible pair in $G$.
	Then, for every $a \in A$ the maximum flow from $s$ to $t$ in $H_a$ is $|A|(2W-1) + 1$.
\end{lemma}

\begin{proof}
	Recall by the definition of a strictly reducible pair that there is a $g$ that satisfy already $\sum_{C \in \B} g(C,a) \geq 2W -1$ for each $a \in A$ and for an $a' \in A$ even $\sum_{C \in \B} g(C,a') \geq 2W$.
	Basically, we want to show that this particular element in $A$ can be chosen arbitrarily if $(A,B)$ is a minimal strictly reducible pair.
	If this is not possible, then we will prove that there is a smaller strictly reducible pair $(A',B')$ with $A' \subset A$ and $B' \subseteq B$ contradicting the minimality of $(A,B)$.
	
	Consider the $g$-embedded flow networks $H$ and $H_a$ for each $a \in A$.
	For each $a \in A$ we search for an augmenting path in the residual graph of $H_a$ with respect to $Z^a$ (in the beginning $Z^a = Z$) from $s$ to $t$, or in other words, we will check, whether a larger maximum $s$-$t$-flow $F=|A|(2W-1)$ as in $H$ is realizable in $H_a$. 
	Observe that an augmenting path in $H_a$ with respect to $Z^a$ has to run through $\ora{at}$. 
	Let $A_1$ and $A_2$ be a partition of $A$, where for $a \in A_1$ the networks $H_a$ allow a greater total flow value from $s$ to $t$ than $H$, i.e., $F_a = F+1 = |A|(2W-1) + 1$.
	Note that if we would have $F_a = F+1$ for all $a \in A$, i.e., $A_1=A$, then the lemma holds and we are done.
	Moreover, note that $A_1 \neq \varnothing$, since 
	before reducing the assignments of $g$ for the embedding in $H$ there is a particular vertex $a \in A$ with $\sum_{C \in \B} g(C,a) \geq 2W$.
	
	Assume $A_1 \neq A$.
	We prove in this case that there is a strictly reducible pair $(A_1,B_1)$ with $B_1 \subset B$, which would contradict the minimality of $(A,B)$.
	Let $C' \in \B'$ be the components in $\B$, where for an $a_1 \in A_1$ we have $z^{a_1}_{\ora{Ca}} > 0$ in $Z^{a_1}$.
	If $N(C') \cap A_2 = \varnothing$ for all $C' \in \B'$, then $(A_1, \bigcup_{C' \in \B'} C')$ is a strictly reducible pair and we are done, since $N(\B') \subseteq A_1$ and we can generate a corresponding assignment function $\B' \times A_1 \to \N_0$ through the flows $Z^{a_1}$ for an arbitrary $a_1 \in A_1$.
	Note that $N(B) \subseteq A$ by the definition of a strictly reducible pair, which in turn means that $\bigcup_{C' \in \B'} C'$ has no neighbor outside $A$.
	Thus, we can assume that there is a $C' \in \B'$ with $N(C') \cap A_2 \neq \varnothing$, let us say $a_2 \in A_2$ is connected to $C'$.
	If $\sum_{a \in A} z_{\ora{C'a}} = \sum_{a \in A} g(C',a) = 0$, then $a_2$ cannot be in $A_2$ as $F_{a_2}=F+1$ is easily realizable as there is a augmenting flow path from $s$ to $t$ via $s,C',a_2,t$ in $H_{a_2}$, and $a_2$ would be in $A_1$.
	Let $a_1 \in A_1$ be a vertex with $z_{\ora{C'a_1}}>0$.
	Observe that the initial flow values $z_{\ora{Ca_1}}>0$ for $C \in \B$ in the network $H_{a_1}$, i.e., before augmenting flow in $H_{a_1}$, can only be decreased if we use the corresponding residual arc $\ora{a_1C}$ in a residual $s$-$t$-path $P$.
	But this cannot happen as $\ora{a_1t}$ is an arc in $P$, which means a residual arc $\ora{a_1C}$ cannot be part of the path $P$.
	From this, we obtain that $z^{a_1}_{\ora{C'a_1}} \geq z_{\ora{C'a_1}}>0$.
	For the following step observe that $z^{a_1}_{\ora{a_1t}} = 2W$ and $z^{a_1}_{\ora{at}} = 2W-1$ for all $a \in A \setminus \{a_1\}$.
	With these facts in hand we can easily transform $Z^{a_1}$ to a flow in $H_{a_2}$ with a total flow value of $F+1$ by subtracting one unit flow from $z^{a_1}_{\ora{C'a_1}}$ and $z^{a_1}_{\ora{a_1t}}$ and adding one unit of flow to $z^{a_2}_{\ora{C'a_2}}$ and $z^{a_2}_{\ora{a_2t}}$.
	However, this contradicts that $a_2 \in A_2$.
\end{proof}

From \cref{lemma::flowEverywhere} we can easily derive the correctness of the \textbf{\cref{lemma::decideWhichA,lemma::SizeNeoghborhood}}, which describes the new properties we found for minimal strictly reducible pairs.

To prove \cref{lemma::PackingSizeLargeEnough}, we use the procedure from \cite{DBLP:conf/esa/Casel0INZ21,DBLP:conf/iwpec/KumarL16} on how to find a ($W+1$)-packing of size $|A|$ in $G[A \cup B]$.
We refine this procedure with respect to the missing vertices of $A \cup B$ using \cref{lemma::flowEverywhere} and the augmenting flow properties of network flows.

\paragraph{\textbf{Proof of \cref{lemma::PackingSizeLargeEnough}:}}
To prove the lemma we will careful pack the vertices in $G[A \cup B] - S$ into connected subgraphs of size $W+1$, such that we obtain a $(W+1)$-packing of size $|A|-|S|+1$.
Let $(A,B)$ be a minimal strictly reducible pair in $G$.
We define for an assignment function $g \colon \B \times A \to \N_0$ an edge-weighted graph $G_g := (A \cup \B, E', w: E' \to \N_0)$, where $E' := \{aC \in A \times \B \mid g(C,a) > 0\}$ and $w_{Ca} = g(C,a)$ for every $Ca \in E'$.
First we demonstrate how a \emph{cycle canceling process} in $G_g$ works without changing the assignments from $\B$ to $A$, i.e., maintaining $\sum_{C \in N\left(a\right)} w_{Ca} = \sum_{C \in \B} g(C,a)$ for all $a \in A$, while ensuring $\sum_{a \in N\left(C\right)} w_{Ca} = \sum_{a \in A} g(C,a) \leq |C|$ for all $C \in \B$.
After that we demonstrate how the \emph{packing process} works to find a $(W+1)$-packing of size $|A|$ in $G[A \cup B]$.
Both processes was demonstrated identically already in~\cite{DBLP:conf/esa/Casel0INZ21,DBLP:conf/iwpec/KumarL16}.
We repeat the explanation of them because they are elementary to understand how we find our desired packing in $G[A \cup B]-S$ of size at least $|A|-|S|+1$.

Let $g$ be the assignment function of the minimal strictly reducible pair $(A,B)$.
Suppose there exists a cycle $Q \subseteq G_g$.
We pick an edge $Ca \in E'(Q)$ with minimum weight $\min_{Ca \in E\left(Q\right)} w_{Ca} = x$.
We remove $Q$ from $G_g$ as follows.
First, we reduce the weight of $w_{Ca}$ by $x$.
Then, traversing the cycle $Q$ from $C$ in direction of its neighbor that is not $a$, we alternate between increasing and decreasing the weight of the visited edge by $x$ until we reach $a$.
In the end we obtain at least one zero weight edge, namely $Ca$, and we remove all edges of weight zero from $G_g$. 
Since $G_g$ is bipartite, the number of edges in the cycle is even.
Thus, we will maintain the weight $\sum_{C \in N\left(a\right)} w_{Ca} = \sum_{C \in \B} g(C,a)$ for all $a \in A$ during this process.
Furthermore, we transfer weight on an edge incident to a $C \in \B$ to another edge, which in turn is also incident to this vertex $C$.
As a result, the condition $\sum_{a \in N\left(C\right)} w_{Ca} = \sum_{a \in A} g(C,a) \leq |C|$ for all $C \in \B$ is still satisfied.
We repeat this process until $G_g$ becomes a forest.
Moreover, if $w_{Ca} > 0$, then $a \in N(C)$ for every $C \in \B$ and $a \in A$, since $g(C,a) > 0$ leads to the same and we add no new edge to $G_g$. 

Next, we give the packing process, i.e., we show how to find the $(W+1)$-packing in $G[A \cup B]$ of size $|A|$ 
We transform accordingly the resulting weights back to a new assignment function $g' \colon \B \times A \to \N_0$ by setting $g'(C,a) = w_{Ca}$.
Note that $g'$ satisfies the condition of \cref{def::RedPair}, and that $G_{g'}$ is a forest.
Let $T$ be a tree in $G_{g'}$ rooted at an $a_r \in A$.
We have $\sum_{C \in \B} g'(C,a) \geq 2W-1$ for every $a \in A_1$ in the tree.
Since $G_{g'}$ is a bipartite graph, every child of $a \in A_1$ is in $C \in \B$ and vice versa.
Let $\B_T$ and $A_T$ be the vertices in $T$ of $\B$ and $A$, respectively.
We construct now the desired packing by mapping the elements of $\B_T$ to $A_T$.
Let $f \colon \B \to A$ be a function that describes for each $C \in \B_T$ to which $a \in A$ it belongs to. 
We set $f(C) = a$ for every $a \in A_T$ if $C \in \B_T$ is a child vertex of $a$.
Note that $a \in N(C)$ if $f(C) = a$.
Thus, for $a_r$ we obtain $\left|\bigcup_{C \in f^{-1}(a_r)} C \right| = \sum_{C \in \B} g'(C,a_r) \geq 2W-1$.
For all other vertices $a' \in A_T$ that are in the tree, we only lose its father vertex $C'$, which in turn has also a father $a'' \in A_T \setminus \{a\}$.
That is, $g'(C',a'') \geq 1$.
It follows that $g'(C',a') \leq |C| - 1 \leq W - 1$ for all those vertices and we obtain $\left|\bigcup_{C \in f^{-1}(a)} C \right| \geq \sum_{C \in \B} g'(C,a) - (W-1) \geq (2W-1)-(W-1) = W$ for all $a \in A_T \setminus \{a_r\}$.
We execute this process with all trees in $G_g$ and get our desired packing $\Q$ by setting each $Q_a = \{a\} \cup f^{-1}(a)$ for $a \in A$ as an element of it.

We know how to find a $(W+1)$-packing $\Q$ of size $|A|$ in $G[A \cup B]$ yet.
As a reminder, the properties for each $Q \in \Q$ are that $|Q| \geq W+1$ and $G[Q]$ is connected (\emph{pack properties}), and that the elements in $\Q$ are pairwise disjoint.
Our goal is now to find a $(W+1)$-packing in $G[A \cup B]-S$ of size at least $|A|-|S|+1$.
Observe by the disjointedness of $\Q$ that after removing $S$ from $\Q$ at most $|S|$ elements in $\Q$ can lose its pack properties and if this is exactly $|S|$, then every $s \in S$ has an exclusive $Q \in \Q$ where it is contained.
Otherwise, either we have ~$|S \cap Q| \geq 2$ for some $Q \in \Q$, or there is an $s \in S$ that has no intersection with an element in $\Q$. Consequently, the elements in $\Q$ containing no vertex from $S$ are at least $|\Q| - |S| + 1 = |A|-|S|+1$ many and form the desired packing satisfying the lemma.

We can assume that every $s \in S$ is contained in exactly one $Q \in \Q$.
Let $S_A = A \cap S$ and $S_B = B \cap S$.
Note by the precondition of the lemma that $S_B \neq \varnothing$ and $S_A \neq A$.
For $s \in S_B$ let $\B(s)$ denote the component in $\B$ where $s$ is contained.
Since every $s \in S$ is contained in a $Q \in \Q$, we can assume for $s \in S_B$ that $\B(s)$ occurs in one of the trees of $G_{g'}$ and therefore $\sum_{a \in A} g'(\B(s),a) > 0$ holds.
To get now the desired ($W+1$)-packing we use the properties of the root node $a_r$ while constructing $\Q$.
Recall that we have $\left|\bigcup_{C \in f^{-1}(a_r)} C \right| = \sum_{C \in \B} g'(C,a_r) \geq 2W-1$ and that we can fix a root node arbitrary from $A$.
Let $a' \in A$ be a vertex with $g'(\B(s'),a) > 0$, which exists as $\sum_{a \in A} g'(\B(s'),a) > 0$.
Let $\Q'$ be the $(W+1)$-packing after the packing process with $f' \colon \B \to A$ as the corresponding assignment where $a'$ is the root in one of the trees of $G_{g'}$.
We can assume that $a' \notin S$ and only $\B(s')$ contains a vertex of $S$ among the child vertices, as  otherwise, we have after the packing process an element $Q \in \Q'$ with $|S \cap Q| \geq 2$ and we are able to find the desired ($W+1$)-packing in $G[A \cup B]-S$. 
If $g'(\B(s'),a) < W$, we obtain $\left|\bigcup_{C \in f'^{-1}(a') \setminus \{\B(s')\} } C  \right|  = \sum_{C \in \B \setminus \{\B(s')\}} g'(C,a_r) \geq 2W-1-(W-1)=W$.
This in turn means that $\left( \{a'\} \cup \bigcup_{C \in f'^{-1}(a')} C \right) \setminus S$ satisfies the pack properties and we are able to construct the desired $(W+1)$-packing, because at most $|S|-1$ elements in $\Q'$ after removing the vertices of $S$ in $\Q'$ can violate the pack properties.

Thus, we can assume that $g'(\B(s'),a') = W$.
If $\sum_{C \in \B} g'(C,a') \geq 2W$, then we are done by the explanations above, since we have $\left|\bigcup_{C \in f'^{-1}(a') \setminus \{\B(s')\} } C  \right|  = \sum_{C \in \B \setminus \{\B(s')\}} g'(C,a_r) \geq 2W-W=W$.
If not, we create a valid assignment function $g''$ like in \cref{def::RedPair} with $\sum_{C \in \B} g''(C,a') \geq 2W$, where $G_{g''}$ is as forest and $\B(s')$ is adjacent to $a'$ in $G_{g''}$. 
For this, we construct a $g'$-embedded flow network $H_{a'}$.
By \cref{lemma::flowEverywhere} we know that there is an augmenting path from $s$ to $t$ and by the construction of $H_{a'}$ that has to cross the residual arc $\ora{a't}$.
This augmenting path $P$ cannot decrease the initial flow values $z^{a'}_{\ora{Ca'}}$ as an augmenting path need to use a residual arc $\ora{a'C}$ to do this, which would contradict that $P$ is a path as it has to use the arc $\ora{a't}$.
In particular, after extracting the resulting assignment function $\hat{g}$ from $Z^{a'}$, i.e., setting $\hat{g}(C,a) = z^{a'}_{\ora{Ca}}$ for all $a \in A$ and $C \in \B$, we have $\hat{g}(\B(s'),a') = W$.
Observe that due to $\sum_{C \in \B} \hat{g}(C,a') \leq |W|$ the vertex $\B(s')$ has degree one in $G_{\hat{g}}$.
That is, if we execute the cycle canceling process in $G_{\hat{g}}$ the vertex $\B(s')$ cannot be in any cycle of $G_{\hat{g}}$ during this process.
As a result, after extracting the resulting assignment function after the cycle canceling process this corresponds to the desired $g''$ where $\B(s')$ is adjacent to $a'$ in $G_{g''}$.

\subsection*{Proof of \cref{lemma::RedPairLPNoSequence}}

To prove \cref{lemma::RedPairLPNoSequence} we use the following property concerning LP's for the $W$-separator problem.

\begin{lemma}
	\label{lemma::partitionLP}
	Let $G=(V,E)$ be a graph and let $V_1,V_2$ be a partition of $V$.
	For every $W \in \N$ of the $W$-separator problem we have $\lpPrim(G) \geq \lpPrim(G[V_1]) + \lpPrim(G[V_2])$.
\end{lemma}    
\begin{proof}
	Let $\{y_v\}_{v \in V}$ be a fractional $W$-separator of $G$.
	For every connected subgraph $G'$ of size $W+1$ in $G[V_1]$ or $G[V_2]$ we have $\sum_{v \in V(G')} y_v \geq 1$. 
	Consequently, $\{y_v\}_{v \in V_1}$ and $\{y_v\}_{v \in V_2}$ are fractional $W$-separators of $G[V_1]$ and $G[V_2]$, which shows that the desired inequality holds.
\end{proof}

First, we give a relation of the LP-objective to a minimal strictly reducible pair $(A,B)$ in case of $B \cap X_1 \neq \varnothing$, which we will use to prove the main statement of this section.

\begin{lemma}
	\label{lemma::RedPairLPNoB}
	Let $(A,B)$ be an minimal strictly reducible pair in $G$ and let $X_1 \in \{0,1\}^n$ be a search point.
	If $X_1 \cap B \neq \varnothing$, then $\lpPrim(G(u[X])) + |X_1| > \lpPrim(G)$.
\end{lemma}
\begin{proof}
	Let $\oAB = V \setminus (\AB)$, $X_1^{\AB} = X_1 \cap (\AB)$ and $X_1^{\oAB} = X_1 \cap \oAB$.
	Note $\AB$ and $\oAB$ are a partition of $V$ as well as $X_1^{\AB}$ and $X_1^{\oAB}$ a partition of $X_1$.
	And recall that for every vertex bipartition $V_1,V_2$ of a graph $G'$ that $\lpPrim(G') \geq \lpPrim(G'[V_1]) + \lpPrim(G'[V_2])$ (cf.~\cref{lemma::RedPairLPNoB}).
	Moreover, observe that $\lpPrim(G[u(X)]) = \lpPrim(G - X_1)$.
	Thus, we have 
	\begin{align*}
		\lpPrim(G[u(X)]) + |X_1|
		&=\lpPrim(G-X_1) + |X_1|\\
		&\geq \lpPrim(G[\oAB] - X_1^{\oAB}) + \lpPrim(G[\AB] - X_1^{\AB}) + |X_1|\\
		&= \lpPrim(G[\oAB] - X_1^{\oAB}) + |X_1^{\oAB}| + \lpPrim(G[\AB] - X_1^{\AB})\\ 
        & \hspace{0.25cm} + |X_1^{\AB}|\\
		&\geq  \lpPrim(G[\oAB]) + \lpPrim(G[\AB] - X_1^{\AB}) + |X_1^{\AB}|,
	\end{align*}
	where the last inequality follows by \cref{lemma::lowerBoundLP}.
	By \cref{lemma::onesInLPReduciblePair} and \cref{lemma::onesInLPSameObjective} we obtain that $\lpPrim(G) = \lpPrim(G-A) + |A|$.
	Note that $\lpPrim(G-A) =  \lpPrim(G[\oAB])$ by the separation properties of $A$ and hence, if we can show that $\lpPrim(G[\AB] - X_1^{\AB}) + |X_1^{\AB}| > |A|$ we are done.
	Clearly, if $|X_1^{\AB}| > |A|$ there is nothing to prove.
	Assume $|X_1^{\AB}| \leq |A|$.
	Let $\{y_v \in \R_{\geq 0}\}_{v \in V}$ be an optimal fractional $W$-separator of $G[\AB] - X_1^{\AB}$.
	Because $B \cap X_1^{\AB} \neq \varnothing$ and $|X_1^{\AB}| \leq |A|$ we know by \cref{lemma::PackingSizeLargeEnough} that $G[\AB] - X_1^{\AB}$ contains a ($W+1$)-packing $\Q$ of size $|A| - |X_1^{\AB}| + 1$.
	This in turn means that $\lpPrim(G[\AB] - X_1^{\AB}) \geq |A| - |X_1^{\AB}| + 1$, since the elements in $\Q$ are pairwise disjoint and for each $Q \in \Q$ we have $\sum_{v \in Q} y_v \geq 1$ for any feasible fractional $W$-separator.
	In summary, we obtain
	\begin{align*}
		\lpPrim (G[u(X)]) + |X_1|  
		&\geq  \lpPrim(G[\oAB]) + \lpPrim(G[\AB] - X_1^{\AB}) + |X_1^{\AB}|\\ 
		&= \lpPrim(G - A) + \lpPrim(G[\AB] - X_1^{\AB}) + |X_1^{\AB}|\\
		&\geq \lpPrim(G - A) + |A| - |X_1^{\AB}| + 1 + |X_1^{\AB}|\\
		&= \lpPrim(G - A) + |A| + 1\\
		&= \lpPrim(G) + 1\\ 
		&> \lpPrim(G).
	\end{align*}
\end{proof}

We are ready to prove \cref{lemma::RedPairLPNoSequence}, which provides a relation of the LP-objective to a sequence of minimal strictly reducible pairs.

\paragraph{\textbf{Proof of \cref{lemma::RedPairLPNoSequence}}}
	Let $A^i := \bigcup_{j=1}^i A_j$.
	By \cref{lemma::onesInLPReduciblePair} and \ref{lemma::onesInLPSameObjective} we have $\lpPrim(G - A^i) = \lpPrim(G - A^{i-1}) + |A_i|$ for every $i \in [m]$.
	Starting with 
	$\lpPrim(G) = \lpPrim(G - A^1) + |A_1|$ and stepwise substituting $\lpPrim(G - A^{i})$ for $i \in [m]$ yields 
	$$\lpPrim(G) = \lpPrim(G - A^m) + |A_1| + \dots + |A_m|.$$
	Thus, if $X_1 = \bigcup_{i=1}^m A_i$, then $\lpPrim(G) = \lpPrim(G[u(X)]) + |X_1|$ proving (\ref{enum::1}).
	
	Suppose $X_1 \cap B_\ell \neq \varnothing$.
	Observe that $|A_i| = \lpPrim(G[A_i \cup B_i])$ by \cref{lemma::onesInLPReduciblePair} and $\lpPrim(G[A_i \cup B_i] - A_i) = 0$.
	Hence, we obtain
	$$\lpPrim(G) = \lpPrim(G - A^m) + \lpPrim(G[A_1 \cup B_1]) + \dots + \lpPrim(G[A_m \cup B_m]).$$
	That is, to guarantee $\lpPrim(G(u[X])) + |X_1| = \lpPrim(G)$ we must have for every $i \in [m]$ that $\lpPrim(G[A_i \cup B_i] - X_1) + |X_1 \cap (A_i \cup B_i)| = \lpPrim(G[A_i \cup B_i])$.
	Note that to compensate $\lpPrim(G[A_i \cup B_i] - X_1) + |X_1 \cap (A_i \cup B_i)| > \lpPrim(G[A_i \cup B_i])$ we need for a $j \in [m] \setminus \{i\}$ that $\lpPrim(G[A_j \cup B_j] - X_1) + |X_1 \cap (A_j \cup B_j)| < \lpPrim(G[A_j \cup B_j])$, however, this is not possible as $\lpPrim(G[A_p \cup B_p] - X_1) + |X_1 \cap (A_p \cup B_p)| \geq \lpPrim(G[A_p \cup B_p])$ for all $p \in [m]$ by \cref{lemma::lowerBoundLP}.
	Finally, we have $\lpPrim(G[A_\ell \cup B_\ell] - X_1) + |X_1 \cap (A_\ell \cup B_\ell)| > \lpPrim(G[A_\ell \cup B_\ell])$ by \cref{lemma::RedPairLPNoB}, which proves (\ref{enum::2}).

\subsection*{Proof of \cref{lemma::CrownStaysCrown}}
\begin{proof}
    We define $V(\B_{A'}) \subseteq B$ for $A' \subseteq A$ as the vertices in the components $\{C \in \B \mid N(C) \cap A' \neq \varnothing\}$. 
	After removing $\hat{A}$ from $A$ and $V(\B_{\hat{A}})$ from $B$ we may derive from \cref{lemma::decideWhichA} that $\left(A \setminus \hat{A},B \setminus V(\B_{\hat{A}})\right)$ is a strictly reducible pair in $G-\hat{A}$.
	This in turn means that there exists a nonempty minimal strictly reducible pair with $A_1 \subseteq (A \setminus \hat{A})$ and $B_1 \subseteq B \setminus V(\B_{\hat{A}})$ in $G-\hat{A}$.
	Since $N(B_1) \subseteq A_1$, we can use similar arguments for $\hat{A}' = \hat{A} \cup A_1$ and $V(\B_{\hat{A}'})$ as for $\hat{A}$ and $V(\B_{\hat{A}})$ to obtain that there exists a minimal strictly reducible pair $(A_2,B_2)$ with $A_2 \subseteq A \setminus \hat{A}'$ and $B_2 \subseteq B \setminus V(\B_{\hat{A}'})$.
	Thus, we can continue this process until $\hat{A} \cup A_1 \cup A_2 \dots $ is equal to $A$.
\end{proof}

\subsection*{Proof of \cref{lemma::XReducedInPop}}
\begin{proof}
	Let $\P$ be a population with respect to $f_2$ in the algorithm \algGlobalSemoAlt.
	By \cref{lemma::firstStepOptLP} we have in expectation a solution $X$ in $\P$ such that $\lpPrim(G[u(X)]) + |X_1| = LP(G)$ and there is an optimal fractional $W$-separator $\{y_v \in \R_{\geq 0}\}_{v \in u(X)}$ of $G[u(X)]$ with $y_v<1$ for every $v \in u(X)$ after $\O(n^3(\log n + \opt))$ iterations.
	
	By \cref{lemma::ingredientParateo1} we have $A_i \subseteq X_1$.
	Observe that $|X_1| \leq \opt$ as $\lpPrim(G[u(X)]) + |X_1| = LP(G)$.
	The algorithm \algGlobalSemoAlt calls with $1/3$ probability the mutation that flips every vertex $X_1$ with $1/2$ probability in $X$.
	That is, reaching a state $X'$ from $X$ with $X'_1 = \bigcup_{i=1}^m A_i$ has probability $\Omega(2^{-\opt})$ and selecting $X$ in $\P$ has probability $\Omega(1/n^2)$.
	From this, in expectation it takes $\O\left(2^{\opt} n^2\right)$ iterations reaching $X'$ from $X$.
\end{proof}

%% file: appendix_apx.tex

\subsection*{Proof of \cref{lemma::apxIsSave}}

To prove \cref{lemma::apxIsSave} we make use of an optimal fractional $W$-separator.
Through such a solution it is possible to find at least one vertex that can be added to a potential $W$-separator with the desired objective value $c>\opt$. 

\begin{lemma}
	\label{lemma::oneDividedByW}
	Let $X \in \{0,1\}^n$ be a search point.
	If $u(X) \neq \varnothing$, then any fractional $W$-separator of $G[u(X)]$ contains a vertex variable with value at least $1/(W+1)$.
\end{lemma}
\begin{proof}
	Let $Y = \{y_v \in \R_{\geq 0}\}_{v \in u(X)}$ be a fractional $W$-separator of $G[u(X)]$.
	Since $u(X) \neq \varnothing$, there is a connected subgraph $G'$ in $G[u(X)]$ with $W+1$ vertices.
	If every $y_v < 1/(W+1)$ for $v \in V(G')$, then $\sum_{v \in V(G')} y_v < 1$, which contradicts the feasibility of $Y$.
\end{proof}  

With \cref{lemma::oneDividedByW} in hand we characterize 0-bit vertices that can be added to potential $W$-separators without affecting the desired target value $c> \opt$.

\paragraph{\textbf{Proof of  \cref{lemma::apxIsSave}:}}
W.l.o.g.~let $X$ be a search point in $\P$ that minimizes $\lpPrim(G[u(X)])$ such that $|X_1| + (W+1) \cdot \lpPrim(G[u(X)]) \leq c$ (\emph{$c$-inequality}). 
If $u(X) = \varnothing$, then $X$ is a feasible $W$-separator with $\lpPrim(G[u(X)]) = 0$ and hence $|X_1| \leq c$.
Otherwise, by \cref{lemma::oneDividedByW} we have a $v \in u(X)$ with $y_v \geq 1/(W+1)$, where $\{y_v \in \R_{\geq 0} \}_{v \in u(X)}$ is a fractional $W$-separator according to the objective $\lpPrim(G[u(X)])$.
If we mutate $X$ by flipping $x_v$ to a search point $X'$, then we obtain $\lpPrim(G[u(X')]) \leq \lpPrim(G[u(X)]) - 1/(W+1)$ and $|X'_1| = |X_1| + 1$.
This implies that
\begin{align*}
	|X'_1| + (W+1) \cdot \lpPrim(G[u(X')]) 
	&\leq |X_1| + 1 + (W+1) (\lpPrim(G[u(X)]) - 1/(W+1))\\
	&= |X_1| +  (W+1) \cdot \lpPrim(G[u(X)])\\
	&\leq c
\end{align*}  
Observe that by the choice of $X$ that $X'$ is not dominated by any other search point in $\P$ and that it can only be dominated by a search point $X''$ where the $c$-inequality holds with $\lpPrim(G[u(X'')]) \leq \lpPrim(G[u(X')])$.
When $X$ mutates to $X'$, the LP-objective of $f(X')$ is at least $1/(W+1)$ smaller than that of $f(X)$.
Since $\lpPrim(G[u(X)]) \leq \opt$, such an event happens at most $(W+1) \cdot \opt$ times.
For the fitness functions $f_2$ and $f_3$, selecting $X$ and flipping only $v$ in it has probability $\Omega(1/n^2)$ or $\Omega(1/n^3)$ in the algorithm \algGlobalSemo (cf.~\cref{lemma::singleFlipAndBoundedPop,corollary::singleFlipAndBoundedPop}), respectively.
By using the method of fitness based partitions \cite{DBLP:conf/ppsn/Sudholt10}, this results in a total expectation time of $\O(n^2 W \cdot \opt)$ or $\O(n^3 W \cdot \opt)$ having the desired search point in $\P$, respectively.

\subsection*{Proof of \cref{thm::approx}}
\algGlobalSemo have the search point $0^n$ in the population $\P$ after $\O(n^2 \log n)$ expected number of iterations (cf.~\cref{corollary::zeroSol}).
This search point satisfies already the precondition of \cref{lemma::apxIsSave} with $c=(W+1) \cdot \opt$, since $|0^n_1| + (W+1) \cdot \lpPrim(G) = (W+1) \cdot \lpPrim(G) \leq (W+1) \cdot \opt$.
That is, once $0^n \in \P$, the algorithms need in expectation $\O(n^2 W \cdot \opt)$ iterations having a search point in $\P$ that is a $(W+1)$-approximation. 
As a result, for both algorithms we have a desired search point in $\P$ after $\O\left(n^2 (\log n +  W  \cdot \opt)\right)$ iterations in expectation.

\subsection*{Proof of \cref{thm::ptas}}
Let $X \in \{0,1\}^n$ be a search point such that $G[u(X)]$ contains no minimal strictly reducible pair (\emph{irreducible-condition}).
Furthermore, let $S$ be an optimal solution of $G[u(X)]$ and let $U = u(X) \setminus S$.
Note that $|S| \leq \opt$.
Since $G[u(X)]$ contains no minimal strictly reducible pair, we have $|U| = |u(X)| - |S| \leq 2W|S| - |S| = |S|(2W-1)$ by \cref{lemma::existenceReducible}.

Recall that \algGlobalSemoAlt chooses with $1/3$ probability the mutation that flips every bit corresponding to the vertices in $u(X)$ with $1/2$ probability.
From this, the search point $X$ has a probability of $\Omega\left(2^{-(1-\varepsilon) |S| - (1-\varepsilon) |S| \cdot (2W-1) }\right) = \Omega\left(4^{-(1-\varepsilon) |S| \cdot W }\right)$ to flip $(1-\varepsilon) |S|$ fixed vertices of $S$ and to not flip $(1-\varepsilon) |S| \cdot (2W-1)$ fixed vertices of $U$ in one iteration.
Independently from this, half of the remaining vertices of $S$ and $U$ are additionally flipped in this iteration, i.e., $\frac{1}{2} \varepsilon |S|$ of $S$ and  $\frac{1}{2} \varepsilon |U|$ of $U$.
Let $S'$ and $U'$ be the flipped vertices in this iteration and let $X'$ be the according search point.
Note that $\lpPrim(X') \leq |S| - |S'|$ simply because there is a $W$-separator of $G[u(X)] - S'$ of size $|S| - |S'|$.
Hence, we have
\begin{align*}
	|X'_1| + (W+1) \cdot \lpPrim(G[u(X')])
	& = |X_1| + |S'| + |U'| + (W+1) \cdot \lpPrim(G[u(X')])\\
	&\leq |X_1| + |S'| + |U'| +  (W+1) \cdot (|S| - |S'|)\\
	&= |X_1| + |S'| + |U'| + |S|(W+1) - |S'|(W+1)\\
	&= |X_1| + |S|(W+1) - |S'|W + |U'|.
\end{align*}
Next, we upper bound $|S'|$ and $|U'|$ in terms of $|S| \leq \opt$.
Using the fact $|U| \leq |S| (2W-1)$, we obtain 
$|U'| = \frac{1}{2} \varepsilon |U| \leq \frac{1}{2} \varepsilon |S|(2W-1) = \varepsilon |S| W - \frac{1}{2} \varepsilon |S|$.
Regarding $S'$ we have
$|S'| = (1-\varepsilon)|S| + \frac{1}{2} \varepsilon |S| = |S| - \varepsilon |S| + \frac{1}{2} \varepsilon |S|$.
As a result, we obtain
\begin{align*}
	& |X'_1| + (W+1) \cdot \lpPrim(G[u(X')])\\
	& \leq |X_1| + |S|(W+1) - |S'|W + |U'|\\
	&\leq |X_1| + |S|(W+1) - (|S| - \varepsilon |S| + \frac{1}{2} \varepsilon |S|)W + \varepsilon |S| W - \frac{1}{2} \varepsilon |S|\\
	&\leq |X_1| + |S|(W+1) - |S|W + \varepsilon |S| W - \frac{1}{2} \varepsilon |S| W + \varepsilon |S| W - \frac{1}{2} \varepsilon |S|\\
	&= |X_1| + |S| + 2 \varepsilon |S| W - \frac{1}{2} \varepsilon |S| W -  \frac{1}{2} \varepsilon |S|\\
	&= |X_1| + |S| + |S|\left(2 \varepsilon W - \frac{1}{2} \varepsilon W  - \frac{1}{2} \varepsilon\right). 
\end{align*}
The next step is to distinguish between (\ref{thm::item1APX}) and (\ref{thm::item2APX}) in the theorem.
The difference is mainly in the selection of $X$ and the upper bounded size of the population $\P$ with respect to the fitness functions $f_2$ and $f_3$.
Note that the chosen $X$ needs to satisfy the irreducible-condition. 
Moreover, observe that once a desired $X$ is guaranteed to be in the population, an event described above occurs in expectation after  $\O\left(n^2 \cdot 4^{(1-\varepsilon) |S| \cdot W }\right)$ and $\O\left(n \cdot 4^{(1-\varepsilon) |S| \cdot W }\right)$ iterations for the fitness functions $f_2$ and $f_3$, respectively, where the probability of choosing $X$ due to population size makes the linear factor difference (cf.~\cref{lemma::singleFlipAndBoundedPop,corollary::singleFlipAndBoundedPop}).
\begin{enumerate}
	\item[(\ref{thm::item1APX})]
	Let $(A_1,B_1),\dots,(A_m,B_m)$ be a sequence of minimal strictly reducible pairs, such that $G-\bigcup_{i=1}^m A_i$ contains no strictly reducible pair. 
	By \cref{lemma::XReducedInPop} we have in expectation a search point $X$ in the population $\P$ with $X_1 = \bigcup_{i=1}^m A_i$ after $\O\left(n^3(\log n + \opt) + 2^{\opt}\right)$ iterations using the fitness function $f_2$.
	Note that $X$ satisfies the irreducible-condition.
	Moreover, $X$ is a pareto optimal solution by \cref{lemma::paretoOptiSol}.
	By \cref{thm::saveReduction} we have $|X_1| = \opt - |S|$.
	Moreover, recall that $|S| \leq \opt$.
	Thus, we obtain
	\begin{align*}
		|X'_1| + (W+1) \cdot \lpPrim(G[u(X')]) 
		&\leq |X_1| + |S| + |S|\left(2 \varepsilon W - \frac{1}{2} \varepsilon W  - \frac{1}{2} \varepsilon\right)\\
		&= \opt - |S| + |S| + |S|\left(2 \varepsilon W - \frac{1}{2} \varepsilon W  - \frac{1}{2} \varepsilon\right)\\
		&\leq \opt + \opt\left(\frac{3}{2} \varepsilon W - \frac{1}{2} \varepsilon\right)\\
		&= \opt\left(1 + \varepsilon\left(\frac{3}{2}W - \frac{1}{2}\right)\right).
	\end{align*}
	As a result, by the choice of $X$ the resulting search point $X'$ satisfies the precondition of \cref{lemma::apxIsSave} with $c= \opt(1 + \varepsilon(\frac{3}{2}W - \frac{1}{2}))$.
	That is, once $X' \in \P$, the algorithm \algGlobalSemoAlt need in expectation $\O(n^3 W \cdot \opt)$ iterations having a search point in $\P$ which is a $(1 + \varepsilon(\frac{3}{2}W - \frac{1}{2}))$-approximation. 
	In summary, the desired search point $X'$ is in $\P$ after $\O\left(n^3(\log n + W \cdot \opt) + 2^\opt + n^2 \cdot 4^{(1-\varepsilon) \opt \cdot W} \right)$ iterations in expectation.
	
	\item[(\ref{thm::item2APX})]
	By \cref{lemma::ingredientParateo1,corollary::firstStepOptLP}, using the fitness function $f_3$, \algGlobalSemoAlt contains after $\O(n^2(\log n + \opt))$ iterations a search point $X \in \P$, such that $|X_1| + \lpPrim(G[u(X)]) = \lpPrim(G)$, and additionally, there is an optimal fractional $W$-separator $\{y_v \in \R_{\geq 0}\}_{v \in u(X)}$ with $y_v < 1$ of $G[u(X)]$ for all $v \in u(X)$.
	Since $y_v<1$ for all $v \in u(X)$, by \cref{lemma::onesInLPReduciblePair} we know that $G[u(X)]$ contains no strictly reducible pair.
	That is, $X$ satisfies the irreducible-condition.
	Furthermore, by  $|X_1| + \lpPrim(G[u(X)]) = \lpPrim(G)$ we have $|X_1| \leq \opt$.
	Thus, using the fact $|S| \leq \opt$ we obtain
	\begin{align*}
		|X'_1| + (W+1) \cdot \lpPrim(G[u(X')])
		&\leq |X_1| + |S| + |S|\left(2 \varepsilon W - \frac{1}{2} \varepsilon W  - \frac{1}{2} \varepsilon\right)\\
		&\leq \opt + \opt + \opt\left(2 \varepsilon W - \frac{1}{2} \varepsilon W  - \frac{1}{2} \varepsilon\right)\\
		&= 2 \cdot \opt + \opt\left(\frac{3}{2} \varepsilon W - \frac{1}{2} \varepsilon\right)\\
		&= \opt\left(2 + \varepsilon\left(\frac{3}{2}W - \frac{1}{2}\right)\right).
	\end{align*}
	As a result, by the choice of $X$ the resulting search point $X'$ satisfies the precondition of \cref{lemma::apxIsSave} with $c= \opt(2 + \varepsilon(\frac{3}{2}W - \frac{1}{2}))$.
	That is, once $X' \in \P$, the algorithm \algGlobalSemoAlt need in expectation $\O(n^2 W \cdot \opt)$ iterations having a search point in $\P$ which is a $(2 + \varepsilon(\frac{3}{2}W - \frac{1}{2}))$-approximation. 
	In summary, the desired search point $X'$ is in $\P$ after $\O\left(n^2(\log n + W \cdot \opt) + n \cdot 4^{(1-\varepsilon) \opt \cdot W}\right)$ iterations in expectation.
\end{enumerate}